\documentclass[a4paper, 11pt]{extarticle}

\usepackage[left=20mm, right=5cm, marginparsep=7mm, marginparwidth=40mm]{geometry}

\usepackage{hyperref}

\usepackage{amsmath,amssymb,amsthm}
\usepackage[utf8]{inputenc}
\usepackage{mathrsfs}
\usepackage{fullpage}
\usepackage{authblk}
\usepackage{xparse}
\usepackage{enumitem}

\usepackage{cite}
\usepackage{mathtools}

\usepackage{graphicx}
\graphicspath{ {images/} }

\usepackage{wrapfig}

\usepackage{dsfont}
\usepackage{tikz}

\usetikzlibrary{arrows, calc}

\usepackage{verbatim}

\usepackage{hyperref}

\usepackage{xcolor}
\hypersetup{
  colorlinks   = true, 
  urlcolor     = {blue!90!black}, 
  linkcolor    = {blue!90!black}, 
  citecolor   = {red!90!black} 
}

\newcommand{\R}{\mathbb{R}}
\newcommand{\N}{\mathbb{N}}
\newcommand{\Z}{\mathbb{Z}}
\newcommand{\Go}{\Gamma _0}

\renewcommand{\preceq}{\lesssim}

\newcommand{\set}[2]{\left\{ #1 \middle| #2 \right\}}

\newcommand{\EE}[1]{\mathbb{E}\left(#1\right)}
\newcommand{\E}{\mathbb{E}}
\newcommand{\Ec}[2]{\mathbb{E}\left[#1 \middle| #2 \right]}
\renewcommand{\P}{\mathbb{P}}
\newcommand{\PP}[1]{\mathbb{P}\left\{#1\right\}}
\newcommand{\m}{\varphi(m)}
\newcommand{\tip}[1]{\text{tip}\left(#1\right)}
\newcommand{\cov}{\mathcal{COV}}
\newcommand{\1}{\mathds{1}}

\newcommand{\pasha}[1]{\color{blue} #1 \color{black}}

\newcommand{\Viktor}[1]{\color{magenta} #1 \color{black}}
\definecolor{indigo}{rgb}{0.29, 0.0, 0.51}
\newcommand{\Viktortwo}[1]{\color{indigo} #1 \color{black}}

\newcommand{\EEp}{L_+^\infty(\R)}
\newcommand{\RR}{\R}

\newcommand{\normom}[1]{\lVert #1\rVert_\varpi}

\newcommand{\eps}{\varepsilon}
\newcommand{\la}{\lambda}

\newcommand{\Tau}{\Upsilon}
\newcommand{\inte}[2]{\int\limits_{#1}^{#2}}

\theoremstyle{plain}
\newtheorem{thm}{Theorem}[section]
\newtheorem{prop}[thm]{Proposition}
\newtheorem{lem}[thm]{Lemma}
\newtheorem{cor}[thm]{Corollary}

\theoremstyle{definition}

\newtheorem{defi}[thm]{Definition}
\newtheorem{rmk}[thm]{Remark}

\theoremstyle{definition}
\newtheorem{defn}[thm]{Definition}

\theoremstyle{remark}
\newtheorem{rem}[thm]{Remark}

\newcounter{assum}

\author[1]{Viktor Bezborodov \thanks{Email: \texttt{integral2012minus1@gmail.com}}} 
\author[1]{
Luca Di Persio \thanks{Email: \texttt{luca.dipersio@univr.it}}}
\author[2]{
 Tyll Krueger \thanks{Email: \texttt{tyll.krueger@pwr.wroc.pl}}}
\author[3]{
Pasha Tkachov \thanks{Email: \texttt{pasha.tkachov@gssi.it}}}

\affil[1]{
{The University of Verona}}
\affil[2]{
{Wroclaw University of Technology}}
\affil[3]
{Gran Sasso Science Institute, L'Aquila}

\title{Spatial  growth processes with long range dispersion:
	microscopics, mesoscopics, and discrepancy in spread rate
 }

\begin{document}

\maketitle

\begin{abstract}
 
 We consider the speed of propagation of a {continuous-time continuous-space}  branching random walk
 with the additional restriction that the birth rate at any spatial point cannot exceed $1$.
The dispersion kernel is taken to have density that decays polynomially as 
$|x|^{- 2\alpha}$, $x \to \infty$. We show that if $\alpha > 2$, then 
the system spreads at a linear speed, {while for $\alpha \in (\frac 12 ,2]$
the spread is faster than linear}. We also consider the mesoscopic equation
corresponding to the microscopic stochastic system. 
We show that in contrast to the microscopic process,
the solution to the mesoscopic equation spreads exponentially fast {for every $\alpha > \frac 12$}.

\end{abstract}



\section{Introduction}

We analyze the truncated pure birth model introduced in  \cite{shapenodeath} on the subject of the speed of space propagation.
Our aim is to approach the question from the microscopic probabilistic as well as the mesoscopic point of views.
It turns out that the scaling significantly changes the behavior of the system: while the microscopic model grows linearly in time
provided the exponent is larger than four, the mesoscopic model spreads exponentially fast.

The limiting behavior of the branching random walk has been extensively studied. 
For an overview of  branching 
random walks and related topics, see e.g. 
\cite{Shi15}.
The asymptotic behavior of the
 position of the rightmost particle
of the branching random walk
under different assumptions
are given in \cite{Dur83}
and \cite{Dur79}, see also references therein.
A shape theorem for a one-dimensional discrete-space supercritical branching random walk with an exponential moment 
can be found in \cite{Big95}; \cite{Big97} contains further comments and extensions,
in particular for a multidimensional branching random walk. 
Further results and references on the branching random walk
with the focus on the position of rightmost  particle
 can be found in \cite{Big_branching_out}.
 More refined limiting properties
have been obtained recently, such as
the limiting law of the minimum or the limiting 
process seen from its tip or the asymptotics of the position of the minima of a branching random walk, 
see  \cite{aidekon2013convergence, aidekon2013branching, arguin2013extremal, Min_of_BRWs}.
For
maximal displacement of branching random walks in an environment 
see e.g. \cite{FZ12, Mal15} and references therein.
A branching random walk with a fixed number of particles 
is treated in \cite{BM14}, where asymptotic properties 
are obtained both in time and in the number of particles.
In \cite{Eth04}, conditions for the survival and extinction
of different versions of the Bolker--Pacala model are given.

Among asymptotic results for other stochastic models,
Blondel \cite{Blo13} proves a shape result and an ergodic theorem for the process viewed from the tip
for the East model.
A continuous-space set-valued stochastic growth model with the related shape theorem was given in   \cite{Dei03}.
The results have been extended in \cite{GM08}. 
The agent based model we treat in the present manuscript shares some features with this set based models.

The transition from the microscopic probabilistic models to macroscopic deterministic evolutions is a subject of several works, see e.g.  \cite{fournier2004microscopic, champagnat2008individual}.
Equations similar to those considered in the present paper appear in
\cite{brunet2007effect} during the analysis of the  rightmost particle of the Branching random walk.
Convolution with a probability density is often considered in biological and ecological models to describe a non-local interaction \cite{CDL2008,LMNC2003}.
Evolution equations involving convolution terms naturally appear as a limiting behavior of rescaled stochastic processes 
\cite{meleard2015some,durrett1988crabgrass,lebowitz1966rigorous,presutti2008scaling,finkelshtein2010vlasov}.
We do not give a formal derivation of the macroscopic model here,
however we show that the microscopic and macroscopic models
may 
have qualitatively different asymptotic growth rate when
the underlying geographic space
is not compact.
 This phenomenon can also be deduced for other models (see Remark \ref{obtuse}).

The main results
 are Theorems \ref{finite speed of propagation}, \ref{thm superlinear growth alpha < 2}, and \ref{thm:est_front}.
Theorem \ref{finite speed of propagation} states that the birth process with the birth rate given by \eqref{motile c} and \eqref{stampede} below propagates not faster than linearly if $\alpha >2$.
We give a proof for  the negative direction only  as the proof for the opposite  direction is identical due to symmetricity.
	Of course, Theorem \ref{finite speed of propagation} also applies to
any stochastic process dominated by the birth process defined in Section \ref{models, assu and res}, see Remark \ref{libel} for more detail.
Theorem \ref{thm superlinear growth alpha < 2} shows that when $\alpha <2$ the birth process does in fact spread faster than linearly.
In combination with Theorem \ref{finite speed of propagation} it allows us to conclude that $\alpha  = 2$
is a critical value for the birth proces defined by  \eqref{motile c} and \eqref{stampede}.  
 On page \pageref{Heuristics on why alpha = 2} 
two heuristic arguments are given on why one could expect the critical value to be two.
In contrast to the linear speed in the stochastic microscopic model for $\alpha >2$, Theorem \ref{thm:est_front} shows that the solution to the respective mesoscopic equation propagates exponentially fast.
Let us note that the effect is different 
for the models
 without restriction: a dispersion kernel with polynomially decaying tails gives
  exponentially fast propagation for both the rightmost particle of the branching random walk 
  (as shown in \cite{Dur83})
   and the unique solution to the corresponding mesoscopic equation (see \cite{FT2017, together_with_FT17, Gar11}).

The paper is organized as follows.
The models we consider, assumptions
and results are collected in 
Section \ref{models, assu and res}.
Proofs of the main results,
Theorems \ref{finite speed of propagation}, \ref{thm superlinear growth alpha < 2},
and \ref{thm:est_front},
are contained in Sections \ref{section micro discr} and \ref{section micro cont},
\ref{section micro alpha less than 2} and \ref{section micro cont},
 and \ref{section meso}, respectively.
 Sections \ref{section micro discr} and \ref{section micro alpha less than 2}
 are devoted to the discrete-space version of the birth process.
Section \ref{section meso} also contains
a remark on heuristic 
connection between the microscopic
and mesoscopic  models.

\section{The model, assumptions and results}\label{models, assu and res}

Let $\Go$ be the collection of
 subsets of finite number of points in $\R^1$,
\[
\Gamma _0(\R ^1)=\{ \eta \subset \R ^1 : |\eta| < \infty \},
\]
where $|\eta|$ is the number of elements in $\eta$.
Let also $b: \R ^1 \times \Go \to \R _+$ be the birth rate
\begin{equation}\label{motile c}
b(x,\eta) = 1 \wedge \left(  \sum\limits _{y \in \eta } a (x-y) \right),
 \ \ \ x \in \R, \eta \in \Gamma _0(\R ^1)
\end{equation}
with
\begin{equation}\label{stampede}
a(z) = \frac{c_\alpha}{(1+|z|^2)^{ \alpha}},\quad z \in \R,
\end{equation}
where $\alpha > \frac 12$ and $c_ \alpha >0$ is such that $\int_{\R} a(z) dz = 1$.
The time evolution can be imagined as follows.
We denote the state of the process at time $t$ by $\eta_t \in \Go$.
If the state of the system is $\eta \in \Go$,
then the rate at which a birth occurs in a bounded Borel set $B$ is $\int_B b(x, \eta)dx$,
that is, the probability 
that a new particle appears (a ``birth'') in 
a bounded set $B\in \mathscr{B}(\R ^1)$
over time interval $[t;t+ \Delta t]$ is 
\[
\Delta t \int\limits _{ B}b(x, \eta)dx + o(\Delta t),
\]
More details can be found in \cite{shapenodeath}. Note that 
the birth rate without restriction 
$$\bar b(x,\eta) = \sum\limits _{y \in \eta } a (x-y)$$
corresponds to a continuous-space branching random walk. 

\begin{thm}\label{finite speed of propagation}
	Assume that $\alpha >2$.
	For the continuous-space birth process $(\eta _t)_{t\geq 0}$ with birth rate 
	\eqref{motile c}
	and initial condition $\eta _0 = \{0\}$
	 there exists  a constant
	$\mathrm{C} _{\alpha} >0$ such that a.s.
	for sufficiently large $t$,
	\begin{equation}
	\eta _t \subset [- \mathrm{C}_{\alpha} t, \mathrm{C}_{\alpha } t].
	\end{equation}
\end{thm}

\begin{rmk}
	As is the case for many shape theorems for  growth models, 
	Theorem \ref{finite speed of propagation} holds true 
	for any initial condition  $\eta _0 \in \Gamma _0(\R ^1) $.
	Also, the upper bound in \eqref{motile c}
	does not have to be $1$, it can be any positive constant.
\end{rmk}

\begin{rmk}\label{shape thm}
	In fact, 
	analyzing the proof 
	of the shape theorem in
	 \cite{shapenodeath},
	 we can obtain a stronger result 
	 for 
	 the one dimensional continuous-space birth process with birth rate 
	 satisfying
	\begin{equation*}
	b(x,\eta) \leq C _{b}  \wedge \left( C _{b} \sum\limits _{y \in \eta } a (x-y) \right),
	\end{equation*}
	for some constant $C _{b}> 0$,
	provided that certain additional conditions
	are satisfied (monotonicity, translation and rotation invariance, 
	and non-degeneracy as defined in \cite{shapenodeath}).
	Specifically, 
	 there exists  a constant
	$\lambda >0$ such that for every $\varepsilon >0$ a.s.
	for sufficiently large $t$ both
	\begin{equation}\label{shape up}
	\eta _t \subset [- \lambda (1+ \varepsilon) t, \lambda (1+ \varepsilon)  t]
	\end{equation}
	and 
	\begin{equation} \label{shape down}
	 \bigcup\limits _{x \in \eta _t} [x-1,x+1] \supset [- \lambda (1- \varepsilon) t, \lambda (1- \varepsilon)  t]
	\end{equation}
hold true.
In particular, \eqref{shape up} 
and \eqref{shape down} hold 
for $b$ defined in \eqref{motile c}
and \eqref{stampede}.
Note that such $b$
does not satisfy
Condition 2.1 from \cite{shapenodeath},
however Condition 2.1 from that paper is only used to
establish that the growth is at most linear,
which we do in a different way in Theorem \ref{finite speed of propagation}. 

\end{rmk}

\begin{rmk}
	Theorem \ref{finite speed of propagation} can be compared 
	with the  result of Durrett \cite{Dur83},
	which shows that we observe  
	an exponential growth for the maximal displacement of a branching random walk with polynomially decreasing 
	dispersion kernel.
	A related result for a branching random walk with
	dispersion kernel satisfying certain semiexponential conditions can be found in 
	\cite{Gan00}. Semiexponential kernels in \cite{Gan00}
	satisfy 
	\[
	\P \{Y \geq t\} = l(t) \exp(-L(t)t^r)
	\]
	for $t$ sufficiently large, where $Y$ is a random variable 
	distributed as displacement of the offspring from the parent, $r \in (0,1)$,
	$l$ and $L$ are slowly varying functions, and $L(t)/t^{1-r}$
	is non-increasing for large $t$. The spread rate for a branching 
	random walk with such a displacement kernel 
	 is given in \cite{Gan00} explicitly. 
	The system grows faster than linearly; for some choices 
	of $L$ the spread rate is polynomial.	  
	For Deijfen's model of a randomly growing set,
	 Gou{\'e}r{\'e} and Marchand \cite{GM08}
	give a sharp condition
	on the distribution of the outburst radii
	for linear or
	superlinear growth (i.e. faster than linear).

\end{rmk}

\begin{rmk}
	In the language of statistical physics,
	Theorem \ref{finite speed of propagation}
	means that our model 
	exhibits
	 the directed percolation (DP) 
	 class properties	 
	while  having longe-range interaction,
	see e.g. \cite[Section 6.7 and elsewhere]{marro_dickman_1999},
	\cite{ long-range_CP_physics, Hin00, Geza04}.

\end{rmk}

\begin{rmk}\label{libel}
	As noted in the introduction, Theorem \ref{finite speed of propagation} also applies to
	 any stochastic process dominated by the birth process with birth rate \eqref{motile c}.
	 In particular, the statement holds true if every particle is removed 
	 after an exponential time with mean $\delta ^{-1}$, that is,
	 if each particle also has a death rate equal to $\delta$. 
\end{rmk}

The next result shows that the condition $\alpha > 2$ in Theorem \ref{finite speed of propagation}
is sharp. The system exhibits a superlinear spread rate when $\alpha \leq 2$.

\begin{thm}\label{thm superlinear growth alpha < 2}
	Assume that $\alpha \in (\frac 12, 2]$. Then $(\eta _t)$ grows faster than linearly 
	in the sense that for any $K_0, K_1 > 0$,
	\begin{equation}\label{superlinear growth formula}
	\PP{ \eta _t \subset [-K_0-K_1 t , K_0 +  K_1 t ] \text{ for sufficiently large t} } = 0.
	\end{equation}
	
\end{thm}
Put differently, \eqref{superlinear growth formula} means that 
any $K_0, K_1 > 0$ the set
	\begin{equation*}
 \Big\{t: \eta_{t} \setminus [-K_0-K_1 t , K_0 +  K_1 t ]  \ne \varnothing\Big\} \subset [0, \infty)
\end{equation*}
 is a.s. unbounded.

 A mesoscopic approximation of the point process $(\eta_t)_{t\geq 0}$ is given by the following evolution equation 
\begin{equation}\label{eq:basic_meso_intro}
\begin{cases}
\begin{aligned}
\dfrac{\partial u}{\partial t}(x,t) =& \ \min\{ \int_{\mathbb{R}} a(x-y)u(y,t)\,dx, 1 \}, &\qquad  &x\in\R,\ t\in (0,\infty),\\
u(x,0) =& u_0(x), &\qquad 	&x\in\R.
\end{aligned}
\end{cases}
\end{equation}
where $a$ is defined by \eqref{stampede}.

It turns out that the mesoscopic model shows a very different behavior.
No matter how large  $\alpha > \frac 12$ is in \eqref{stampede},
the speed of propagation is faster than linear as we
see in Theorem \ref{thm:est_front} which states that the solution to \eqref{eq:basic_meso_intro} propagates exponentially fast.
Moreover, solutions with roughly speaking `monotone' initial conditions (case 2) propagate faster than solutions with `integrable' initial conditions (case 1). 
\begin{thm}\label{thm:est_front}
	Let ${0\leq u_0\in L^\infty(\R)}$  and $u=u(x,t)$ be the corresponding classical solution to \eqref{eq:basic_meso_intro} with $a(x)$ defined by \eqref{stampede}. 
	Then, for any $\eps\in(0,1)$, $n\geq1$, there exists $\tau=\tau(\eps,n)$ such that the following inclusions hold
	\begin{enumerate}
		\item If there exists $C>0$ such that $u_0(x) \leq Ca(x)$, $x\in\RR$, and there exist $\mu>0$, $x_0\in \R$, such that $u_0(x) \geq \mu$, $x\in [x_0-\mu, x_0+\mu]$, then for all $t\geq \tau$,
		\begin{equation}\label{eq:front_position_u_0_integrable}
		\{ x: u(x,t)\in [\tfrac{1}{n},n] \} \subset \{x:  e^{\frac{1-\eps}{2\alpha}t} \leq |x| \leq e^{\frac{1+\eps}{2\alpha}t} \}.
		\end{equation}
		\item If there exists $C>0$ such that $u_0(x) \leq C \inte{x}{\infty} a(y)dy$, $x\in\RR$, and there exist $\mu>0,\ \rho\in\R,$ such that $u_0(x) \geq \mu,\ x\leq \rho,$ then for all $t\geq \tau$, 
		\begin{equation}\label{eq:front_position_u_0_decreasing}
		\{ x: u(x,t)\in [\tfrac{1}{n},n] \} \subset \{x : e^{\frac{1-\eps}{(2\alpha-1)}t} \leq x \leq e^{\frac{1+\eps}{(2\alpha-1)}t} \}.
		\end{equation}
	\end{enumerate}
\end{thm}

\begin{rmk}
	We use the term `mesoscopic approximation' here instead of `macroscopic approximation', even though some authors might use the latter to describe \eqref{eq:basic_meso_intro}.
	We follow here \cite{presutti2008scaling}; see also \cite{lachowicz2011microscopic}, \cite{saragosti2010mathematical} for discussions of microscopic, mesoscopic and macroscopic descriptions of complex systems.
\end{rmk}

\begin{rmk}\label{obtuse}
	Remark \ref{libel} can also be contrasted with the spread rate of
	the system driven by the equation
	\begin{equation} \label{expiate}
	u_t  = J * u - u + f(u)
	\end{equation}
	where $J$ is the dispersion kernel, $\|J\|_{L^1} = 1$, 
	and $f:[0,1] \to \R _+$ is some differentiable function with $f(0)= f(1) = 0$ and $f'(0)> 0$,
	and certain other mild conditions.
	It is shown in \cite{Gar11} that the solution 
	to \eqref{expiate} has level sets moving faster than linearly.
	We note that since the solution to \eqref{expiate} takes values between $0$ and $1$ (provided
	that the initial condition lies between $0$ and $1$; see \cite{Gar11}), we have $J * u \leq 1 $,
	and hence \eqref{expiate} can be written as 
	\begin{equation}
	u_t  = \min \{1, J * u \} - u + f(u).
	\end{equation}

\end{rmk}

\emph{Notation and conventions}.
Let $\R_+ = [0, \infty)$, $\R_- := (-\infty,0]$ and $\Z_+ = \{m \in \Z: m \geq 0\}$.
For processes indexed by  $\R _+$ (which represents  time) we will use  $( X_t)$ as a  shorthand for $( X_t)_{t \geq 0}$ or $ \{X_t, t\geq 0 \}$.
For a Poisson process $(N_t)$, $0<a\leq b$, $N(a,b] = N_b - N_a$ and $N(\{a\}) = N_{a} - N_{a-}$. \label{n in a point}
For $a, b \in \R$, $a_+ = \max \{a, 0 \}$, $a\vee b = \max \{a,b\}$, $a\wedge b = \min \{a,b\}$.
Concening the operation order, 
we take
for $a, b, c \in \R$, $- a \wedge b = - (a \wedge b)$, $ab \wedge c = (ab)\wedge c$,
and the same rules for $\vee$.
$\mathrm{Cov}(X, Y)$ and $\mathrm{Var}(X)$ denote the covariance between $X$ 
and $Y$ and variance of $X$, respectively.
$\1$ is an indicator, for example 
\[
 \1 \{x \geq 0\} = 
 \begin{cases}
  1 \text{ if } x\geq 0, \\
  0 \text{ if } x< 0.
 \end{cases}
\]
Throughout the paper, $C$ denotes different universal constants whose exact values are irrelevant.
Even in the concatenation
\[
 F \leq C G \leq C H,
\]
where $F, G$, and $H$ 
are some expressions, two occurrences 
of $C$ may have different values.
We set $B_r(x) = \{y\in\RR\, \vert \, |x-y|\leq r\}$ and $B_r = B_r(0)$.
For simplicity of notations we will write $``x\in\RR"$ instead of $``a.e.\,\,x\in\RR"$ for the elements of $L^\infty(\R)$.
We denote 
\begin{equation}\label{def:Lp_plus}
\EEp = \{ f \in L^\infty(\R) \vert f(x) \geq 0,\ x\in \RR;\ \exists\delta>0,\ x_0\in\R:\ f(x)\geq\delta,\ x\in B_\delta(x_0) \}.
\end{equation}
We will write for $f_1,f_2\in L^\infty(\R)$, $A\subset \RR$,
\[
	f_1(x) \preceq f_2(x),\quad x\in A,
\]
if there exists $c>0$ such that $f_1(x) \leq c f_2(x),\ x\in A$.
For $p\in[1,\infty]$, $\|\cdot\|_p := \|\cdot\|_{L^p(\R)}$.

\emph{A very brief outline of the proof of Theorem \ref{finite speed of propagation}}.
The proof of  Theorem \ref{finite speed of propagation} is split across 
Sections \ref{section micro discr} and \ref{section micro cont}. 
The main bulk of the proof is carried out in  
 Section \ref{section micro discr}, where we prove the equivalent of 
Theorem \ref{finite speed of propagation} for the case when the underlying 
`geographical' space is discrete $\Z ^1$ rather than continuous $\R ^1$.
This equivalent is given in Theorem \ref{at most linear}, and
 Sections \ref{section micro discr} is entirely devoted to the proof of 
 Theorem \ref{at most linear}.
 The main idea of the proof is a coupling of the process seen from its tip with 
 a simpler process. Some of the ingredients are the strong law of large numbers for dependent random variables,
 a form of the strong law  for martingales, and
 Novikov's inequalities, or Bichteler-Jacod's inequalities,  for discontinuous martingales. 
 A brief outline of  the proof of Theorem \ref{at most linear} and Section \ref{section micro discr} 
 can be found on Page \pageref{A very brief summary of Section 3}.
 In Section \ref{section micro cont} we finally prove Theorem \ref{finite speed of propagation}
 by coupling the continuous-space process with the discrete-space
 process from Section \ref{section micro discr}.

\section{Lattice truncated process. Linear growth for $\alpha > 2$}\label{section micro discr}

In this section we introduce a discrete-space equivalent defined by \eqref{motile} and \eqref{toad}
 for our continuous-space process 
defined by \eqref{motile c} and \eqref{stampede}. 
We prove in this section that this discrete-space process
spreads not faster than linearly (Theorem \ref{at most linear}).

We consider the birth process  on $\Z _+ ^{\Z}$
with the birth rate
\begin{equation}\label{motile}
b^{(d)}(x,\eta) = 1 \wedge \left( \sum\limits _{y \in \Z  } 
\eta(y) a^{(d)}(x-y) \right), \ \ \ x \in \Z ^1, \eta \in \Z _+ ^{\Z}.
\end{equation}
where
\begin{equation}\label{toad}
a^{(d)}(x) =  \frac{1}{(1 \vee |x|^2)^{\alpha} },
\end{equation}
 (for convenience we consider a slightly modified $a$ in this section
 compared to \eqref{stampede})
 and the initial condition
\[
	\eta _0(k) = \1 \{ k = 0 \}, \quad k \in \Z.
\]

Thus, if the state of the system is $\eta$, the birth  at $x \in \Z$
(that is, the increase by $1$ of the value at $x$) occurs at rate $b^{(d)}(x,\eta) $.
In this section
we denote the resulting birth process
by $(\eta _t)$.
The process is constructed from a Poisson point process
as a unique solution to a certain stochastic equation
as described below.

Note that since $a^{(d)}(x)\leq 1$ for all $x$,
\begin{equation}\label{detriment}
 b^{(d)}(x,\eta) = 1  \quad \quad  \text{ if } 
 \quad  \eta (x) \geq 1.
\end{equation}

Our aim now is to show that the process propagates not faster than at a finite speed if $\alpha >2$.
Throughout this session we assume  $\alpha >2$.
To this end we introduce the process $(\xi _t)_{t \geq  0}$ as $(\eta _t)_{t \geq  0}$ seen from its left tip.

\begin{defi}
	Define $\xi _t(k)  = \eta _t ( \tip{\eta _t} + k)$, $k \geq 0$, where 
	\[
 		\tip{\eta } = \min\{n : \eta  (n) > 0 \}.
	\]
\end{defi}

Note that $(\xi _t)$ takes values in $\Z _+ ^{\Z _+}$.
Now we introduce another process taking values in $\Z _+ ^{\Z _+}$. We will see later
that this process dominates $(\xi _t)$
in a certain sense specified below.

\begin{defi}\label{defizeta}
  Let $(\zeta _t)$ be a process on $\Z _+ ^{\Z _+}$ evolving as follows.
	The process starts from $\zeta _0 (x) = \1 \{x \geq 0 \}$ and 
  \begin{itemize}
   \item at rate $1$ the configuration is shifted to the right by $1$
   and  a particle is added at zero; that is, 
   if a shift occurs at $t$ and $\zeta _{t-} \in \Z _+ ^{\Z _+}$
   is the state before the shift,
   then 
   \[
   \zeta _t  (k) = \zeta _{t-} (k - 1), \ \ \  k \in \N,
   \]
  and $\zeta _t (0) = 1$.

   \item  between the shifts, $\zeta _t (k)$, $k \in \Z_+$,
   evolves as a Poisson process.
   The Poisson processes 
   are independent for different $k$ and of shift times.
  \end{itemize}
\end{defi}

  \label{Heuristics on why alpha = 2}

  \emph{Some heuristics on why the critical value $\alpha _c $ is two. }
  As was noted above, Theorem \ref{finite speed of propagation}
  and Theorem \ref{thm superlinear growth alpha < 2}
  allow us to conclude that for the birth process with rate
  given by \eqref{motile c} the critical value of $\alpha$ is two.
  In this section we prove that the growth is linear for the discrete-space equivalent 
  model and $\alpha > 2$. Before describing the proof, we take a brief pause 
  and give a few heuristic arguments on why the critical value is $2$.
  We give here two arguments, the first one being shorter
  and possibly more straightforward,  while the second one 
  relying on a heuristic comparison to other models.
  
    We start from the following observation.
    Since we take a minimum with $1$ in \eqref{motile}, it is to be expected 
  that, provided that the spread is linear, $\eta _t$ seen from 
  its tip satisfies for some $c >0$
  \begin{equation} \label{commandment}
  \E \xi _t (k) \approx c k,  \ \ \ k \in \N,
  \end{equation}
  or at least 
  \begin{equation}\label{commandment weaker}
  \frac 1t \int\limits _0 ^t \E \xi _s (k) \approx c k,  \ \ \ k \in \N, t \to \infty.
  \end{equation}

  We now proceed with  the first shorter argument. 
  If the system spreads linearly in time, then we can expect that 
  \eqref{commandment} holds.
  Let $X_t = - \tip{\eta _t}$ be the distance from the leftmost occupied site to the origin.  
   The  rate at which $X_t$ 
  jumps by $k$
  is 
  $$j_k (t) = j_k =   \sum\limits _{i = 0} ^{\infty}  a^{(d)} (-k-i)  \xi _{t-} (i) . $$
  For the speed of propagation to be finite we need 
  the sum
  $ {\sum\limits_{ k =1 } ^\infty k j_k} $ to be finite
  (more precisely,  
  the time averages of ${ \sum\limits_{ k =1 } ^\infty k j_k }$ need to be finite and growing not faster than linearly in time;
  note that  ${X_t - \int\limits _0 ^t \sum\limits_{ k =1 } ^\infty k j_k (s-) ds} $ is a martingale).
  Substituting $\xi _{t-} (i)$
  by $c i$ as in
  \eqref{commandment}, we get 
  \begin{equation}\label{last one?}
  \begin{gathered}
  \sum\limits_{ k =1 } ^\infty k j_k =
  c \sum\limits_{ k =1 } ^\infty \sum\limits _{i = 0} ^{\infty}  \frac{ik}{(i+k) ^{2 \alpha}} 
  = c \sum\limits _{m = 1} \frac{1}{m^{2\alpha}} \sum\limits _{ k = 1} ^ m k (m-k)
  \sim  \sum\limits _{m = 1} \frac{m ^3}{m^{2\alpha}} ,
  \end{gathered}
  \end{equation}
  where $\sim $ means that two series have the same convergence/divergence properties.
  We see that the sum in \eqref{last one?} is finite if and only if $\alpha >2$,
  hence one could expect that the critical value $\alpha _c = 2$.
  
  To make the first heuristic argument rigorous we would have to prove something like
  \eqref{commandment} or \eqref{commandment weaker}. 
  However,
  to prove 
  \eqref{commandment} or \eqref{commandment weaker} 
  we 
  would probably 
  need to prove the linear spread rate first.
  In the actual proof that $\alpha = 2$ is critical we 
  dominate $(\xi _t) $ by another process 
  satisfying a 
  weaker version of \eqref{commandment}.
  This auxiliary process helps us  derive
   an inequality giving an  upper bound
    for certain time averages of $(\xi _t) $, see Proposition \ref{dibber}.
  
  The second argument is of purely heuristic nature.
  We introduce two more birth rates,
  \begin{equation}\label{b^d1}
  b^{(d,1)}(x,\eta) = 
   a^{(d)}(x- \tip{\eta }) , \ \ \ x \in \Z ^1, \eta \in \Z _+ ^{\Z},
  \end{equation}
  and 
    \begin{equation}\label{b^d2}
  b^{(d,2)}(x,\eta)  = \sum\limits _{y = \tip{\eta }  } ^0 
   a^{(d)}(x-y) , \ \ \ x \in \Z ^1, \eta \in \Z _+ ^{\Z}.
  \end{equation}
      Denote by $(\eta _t ^{(d,i)})$  the respective birth processes, $i = 1,2$,
 and by $X _t^{(d,i)} =  -\tip {\eta ^{(d,i)}}$ the distance from the leftmost occupied
 site to the origin.
  For $\eta \in \Z _+ ^{\Z}$ with $\tip{\eta}$ well defined,
  let the `essential parts' of the configuration be
  \begin{align*}
    \tilde \eta  ^{(d,1)} (k) & =  \1 \{ k = \tip{\eta }  \}.
  \\
  \tilde \eta  ^{(d,2)} (k) & =  \1 \big\{ k \in \{0, -1, ..., \tip{\eta } \} \big\}.
  \end{align*}

  Note that
    \begin{equation}
  b^{(d,i)}(x,\eta) = 
   b^{(d,i)}(x,\tilde \eta ^{(d,i)}), \ \ \ x \in \Z ^1, \ \eta \in \Z _+ ^{\Z}, \ i =1,2,
  \end{equation}
  so to determine
   the spread rate of $(\eta _t ^{d,i})$
   it is sufficient to know only $(\tilde \eta _t  ^{(d,i)}) $.

   From the definition of $b^{(d,1)}(x,\eta)$
  we see that 
   $X^{(d,1)}$ is a continuous-time discrete-space
  random walk with jumps by $n\in \N$ occuring at rate $a^{(d)}(n)$.
  Therefore,  for $b^{(d,1)}$ the critical value separating 
  linear and superlinear growth is $\alpha = \alpha _c ^{(1)} = 1$ 
  in
  \eqref{toad}.
  
  Now, it is not as straightforward  to
  determine the critical value  $\alpha _c ^{(2)} $
  for $(\eta _t ^{(d,2)})$. We note however that $(\tilde \eta _t ^{(d,2)})$
  is a discrete-space equivalent of the Deijfen's model \cite{Dei03, GM08}.
  It was shown in \cite{GM08} that in one dimension 
  the critical exponent in the kernel is three. Hence it should hold $\alpha _c ^{(2)}  = \frac 32 $.
  
  Let us come back to $(\eta _t)$ with birth rate \eqref{motile} and
   compare $(\eta _t ^{(d,1)})$, $(\eta _t ^{(d,2)})$, and $(\eta _t )$. 
   We start by noting that all three processes are  related
   because they are defined in terms of $a^{(d)}$.   
   The essential part of $(\eta _t ^{(d,1)})$ is a single site 
   $\tilde \eta _t  ^{(d,1)} $. We can roughly say that 
   the essential part of $(\eta _t ^{(d,1)})$ has dimension zero.
   The critical exponent for $(\eta _t ^{(d,1)})$
   is two, which corresponds to the critical value $\alpha _c ^{(1)} = 1$.
   The essential part of $(\eta _t ^{(d,2)})$, 
   namely $\tilde \eta _t  ^{(d,2)} $, can be thought of 
   as a growing  interval. Thus, informally,
   the essential part of $(\eta _t ^{(d,2)})$ has dimension one.
   The respective critical exponent is three, corresponding to 
   the critical value $\alpha _c ^{(2)} = \frac 32 $.
   
   Now, the essential part of $\eta _t$ is $\eta _t$ itself,
   since every site affects the birth rates beyond the tip.
  The number of occupied sites for $(\eta _t )$
  grows at least linearly with time. 
  According to \eqref{commandment},
  the number of particles at each occupied sites 
  also grows linearly.
  Thus, roughly speaking, the essential part of  $(\eta _t)$
  has two dimensions. 
  We can then conjecture that the critical exponent for $(\eta _t)$
  should be one more than that for $(\eta _t ^{(d,2)})$, to compensate for the one more dimension
  (see also Table \ref{table}), and hence $\alpha  _c = 2$.
  
  \begin{table}[!h]
  	\label{table}
  	  \begin{center}
  	\begin{tabular}{ |c|c|c|c| } 
  		\hline
  		The process          & Dimension of the essential part & Critical exponent & Critical value of $\alpha$ \\ 
  		\hline\hline
  		$(\eta _t ^{(d,1)})$ & 0                               & 2                 &   1                        \\ 
  		\hline
  		$(\eta _t ^{(d,2)})$ & 1                               & 3                 &         $\frac 32$         \\ 
  		\hline
  		$(\eta _t )$         & 2                               & ???               &          ???                \\ 
  		\hline
  	\end{tabular}
  \caption{ The critical exponents and the essential dimension }
  \end{center}
    \end{table}
  
  Of course, for  the above heuristic argument to work it is necessary
  also to assume
   that the restriction in $b$
  given by \eqref{motile}, i.e. taking minimum with $1$, 
  does not affect the sites beyond the  tip too much. This seems to be plausible,
  at least for the sites far away from the tip, while the sites near the tip should
  not affect the critical value too much.
  
  Knowing that the guess $\alpha _c = 2$ is correct, we can go a little bit further
  and conjecture that 
  for this kind of models
  \begin{equation}\label{crit exp = essdim + 2}
  \text{Critical exponent} = 
   \text{ Dimension of the essential part} + 2.
  \end{equation}
  Note that this is compatible with the results 
  of \cite{GM08} as the `essential part' of their  $d$-dimensional model would have dimension $d$ as well.
  Let us add that \eqref{crit exp = essdim + 2} is also compatible
  with the discussion of the DP regime
  for the one-dimensional long-range contact process
   in \cite[Page 6 and elsewhere]{long-range_CP_physics}
   because the `essential part' of the  contact process conditioned on
   non-extinction has dimension one.  
  For the  birth process in $d$ dimensions with birth rate
  as in  \eqref{motile}, \eqref{crit exp = essdim + 2} would mean that the critical value is
  \begin{equation*}
   \alpha  = \frac{ d + 3}{2}.
  \end{equation*}

	\label{A very brief summary of Section 3}
 \emph{A brief summary of the section}. As mentioned above, this section is devoted to proving
 that the  distance $(X_t)$ from the origin to the leftmost particle of  $(\eta _t)$ 
 does not grow faster than linearly in time, as formulated in Theorem \ref{at most linear}.
 We rely on the representation $X_t = Q_t + M_t $,
 where $(Q_t)$ is a suitable increasing process, 
 and
  $(M_t)$ is a local martingale later shown to be a true martingale, see
 \eqref{racy}, \eqref{racy2}, and Lemma \ref{semblance}. 
 
 We then proceed to show that a.s. $(Q_t)$ grows not faster than linearly in time 
 as stated
 Proposition \ref{dibber}. To prove Proposition \ref{dibber},
 we introduce in \eqref{def Y_n} a sequence of random variables $\{Y_n\}_{n \in \N}$
   dominating $\{Q_n\}_{n \in \N}$. The sequence $\{Y_n\}_{n \in \N}$
 is closely related to $(\zeta _t)$ while $\{Q_n\}_{n \in \N}$
 is related to $(\xi _t)$,
 and we make use 
 of the fact that
  the process
 $(\zeta _t)$ stochastically dominates $(\xi _t)$
 in the sense made precise below, see Definition \ref{stoc dom} and Proposition \ref{astute}.
 We then proceed to show that $\{Y_n\}_{n \in \N}$
 grows not faster than linearly with $n$.
  A key point in this step
 is a certain decorrelation property \eqref{rebut}, 
 which we establish using properties of
  $(\zeta _t)$.
 Thanks to \eqref{rebut} we are able to apply to $\{Y_n\}_{n \in \N}$ a strong
 law of large numbers  for dependent random variables,
 concluding the proof of Proposition \ref{dibber}.

	Then, using  representation \eqref{a presto sim a dopo} for $(X_t)$
	and Novikov's inequality for discontinuous martingales, we obtain
	a moment estimate for $(M_t)$ in Proposition \ref{fluty}.
	This moment estimate allows us to apply a strong law of large numbers
	for martingales formulated in Theorem \ref{thm 2.18 HH80}.
	
 	By that point we have practically  shown that $(Q_t)$ grows at most linearly in time
 	and $ \frac{M_n}{n} \to 0$, $n \in \N$.
 	This allows us to conclude in Theorem \ref{at most linear}
 	  that $X_t = Q_t + M_t$ does not grow faster than linearly either.

\begin{defi}\label{stoc dom}
	We say that a random element $R_2$ taking values in $\Z _+ ^{\Z _+}$ \emph{stochastically dominates} a random (again $\Z _+ ^{\Z _+}$-valued) element $R_1$ if
	a.s. for every $k = 0,1,...$
  \begin{equation}\label{balkanization}
    \sum\limits _{i = 0} ^k R_1(i) \leq 
    \sum\limits _{i = 0} ^k R_2(i).
  \end{equation}
  We will say that a process $(\hat \zeta _t)$ \emph{stochastically dominates} another process $(\hat \xi _t)$ if a.s. for every $t$ and every $k = 0,1,...$
  \begin{equation} \label{permeate}
    \sum\limits _{i = 0} ^k \hat \xi _t(i) \leq 
    \sum\limits _{i = 0} ^k \hat \zeta _t(i).
  \end{equation}
\end{defi}

The following lemma is a straightforward consequence of Definition \ref{stoc dom}.
  \begin{lem}\label{rambunctious kitty}
   Let $\{a_i\}_{i \in \Z _+}$
   be a non-increasing sequence
   of non-negative numbers.
   If $R_2$ 
   stochastically dominates 
   $R_1$, both are $\Z _+ ^{\Z _+}$-valued
   random elements,
   then
   \begin{equation}\label{irate}
    \E \sum\limits _{i \in \Z _+}
    a_i R_1 (i) \leq 
     \E \sum\limits _{i \in \Z _+}
    a_i R_2 (i).
   \end{equation}
  In particular, if the right hand side of
  \eqref{irate} is finite,
  then so is the left hand side.
  \end{lem}

\emph{Construction and coupling of $(\eta _t)$, $(\xi _t)$,
	and $(\zeta _t)$}.
Here we construct 
the processes $(\eta _t)$, $(\xi _t)$,
and $(\zeta _t)$
in such  a way that $(\zeta _t)$ 
stochastically dominates $(\xi _t)$. 
We start with $(\eta _t)$, which in this section is the discrete
space birth process with birth rate given by \eqref{motile} and \eqref{toad},
and in whose behavior we are interested in. The processes $(\xi _t)$
and $(\zeta _t)$ are auxiliary processes we need to analyze the 
position of the leftmost occupied site of $(\eta _t)$.

Let $\mathbf{N}$ be a Poisson point process 
on $\R _+ \times \Z \times [0,1]$
with mean measure $ds \times \# \times du$,
where $\# $ is the counting measure on $\Z$. 
Then $(\eta _t)$ can be defined 
as the unique solution to the equation 
(see \cite[Section 5]{shapenodeath})
\begin{equation} \label{se}
 \eta _t (k) = 
 \int\limits _{(0,t] \times \{k\} \times [0, 1 ]}
 \1_{ [0,b^{(d)}(i, \eta _{s-} )] } (u)
 \mathbf{N}(ds di du)
 + \eta _0 (k),
\end{equation}

Define a filtration of $\sigma$-algebras 
$\{ \mathscr{F} _t, t\geq 0 \}$ as 
the completion of
 \begin{align}\label{nonchalant}
   \mathscr{F} ^{0} _t =  \sigma \bigl\{ 
\mathbf{N}(B_1 \times \{ k \} \times B_2) ,
 B_1\in \mathscr{B} ([0,t]), k \in \Z,
B_2\in \mathscr{B} ([0,1])  \bigr\},
 \end{align}
The filtration $\{ \mathscr{F} _t, t\geq 0 \}$ 
is right-continuous and complete.
All the stopping times we consider 
in this section
are with respect to this filtration.

Let $\{N^{(j)} \}_{j \in \Z}$ be a collection of independent 
Poisson processes indexed by $\Z $
defined by \label{a def}
\[
 N^{(j)} _t = \mathbf{N}([0,t]\times \{ j \} \times [0,1])
\]
and  let $\{u^{(j)} _i \}_{j,i \in \N}$
be a two-dimensional array of independent uniformly distributed on $[0,1]$ random variables uniquely defined by 
\begin{equation}\label{foolhardy}
 \mathbf{N}( \{t^{(j)} _i\} \times \{j \}
\times \{ u^{(j)} _i \} ) =1, 
\end{equation}
where 
$t^{(j)} _i = \inf\{t > 0 : \mathbf{N} ([0,t] \times \{j \} \times [0,1] ) = i\}$.
Note that the processes $\{N^{(j)} \}_{j \in \Z}$ and $\{u^{(j)} _i \}_{j,i \in \N}$ are mutually independent.

The evolution of $(\xi _t)$ can be described in terms of $\{N^{(j)} \}_{j \in \Z}$ and $\{u^{(j)} _i \}_{j,i \in \N}$ as follows.
Shifts by $m \in \N$ to the  right occur at moments $t$ when 
$N^{(\tip{\eta} - m)}(\{t\} \times [0, \sum\limits _{k \geq 0} \frac{\xi _{t-}(k)}{(k + m)^{2\alpha }} ]) = 1$,
 and a particle at zero is added.  Between the shift times, the number of particles  at a site
  $j$ grows according to $N^{(\tip{\eta} + j)}$ for $(\xi _t)$; however, an increment by $1$ at time $t$ 
   at the site $j$ actually occurs if not only $N^{(\tip{\eta} + j)}_{t} - N^{(\tip{\eta} + j)}_{t-} = 1$, but also additionally 
\begin{equation} \label{sass}
 u ^{(\tip{\eta} + j)} _{N^{(\tip{\eta} + j)}_t} \leq 
 \sum\limits _{k \geq 0} 
\frac{\xi _{t-}(k)}{(1 \vee |k - j|)^{2\alpha }}.
\end{equation}
If \eqref{sass} is not satisfied, then the value stays the same: $\xi ^{(\tip{\eta} + j)} _t = \xi ^{(\tip{\eta} + j)} _{t-}$.
Thus,  $(\xi _t)$ is a $\Z _+ ^{\Z _+}$-valued process started from $\xi_0 (k) = \1 \{ k = 0 \}$,
$k \in \Z _+$,  that can be described by the following list of events:
\begin{itemize}
	\item for $m \in \N$, shifts by $m$ occur 
	at rate  $ \sum\limits _{k \geq 0} \frac{\xi _{t-}(k)}{(k + m)^{2\alpha }} $.
	Whenever a shift occurs, a single particle is added at the origin.
	(this event occurs at moments $t$ when 
	$N^{(\tip{\eta} - m)}(\{t\} \times [0, \sum\limits _{k \geq 0} \frac{\xi _{t-}(k)}{(k + m)^{2\alpha }} ]) = 1$ );

	\item  the number of particles  at a site
	$j$ increases by $1$ at  rate $1 \wedge	\sum\limits _{k \geq 0} 
	\frac{\xi _{t-}(k)}{(1 \vee |k - j|)^{2\alpha }}$
	( the increase by $1$ occurs at the
	 jump times of $N^{(\tip{\eta} + j)}$ provided that additionally
	$ u ^{(\tip{\eta} + j)} _{N^{(\tip{\eta} + j)}_t} \leq 
	\sum\limits _{k \geq 0} 
	\frac{\xi _{t-}(k)}{(1 \vee |k - j|)^{2\alpha }}$ );
	
	\item The above events happen independently, and no two events occur at the same time.
	
\end{itemize}

Let us now define $(\zeta _t)$ in terms of $\{N^{(j)} \}_{j \in \Z}$. 
Recall that the initial configuration is $\zeta _0 (  k) =1 $, $k \in \Z _+$. 
A shift by $1$ occurs at time moments $t$ when $ N^{(\tip{\eta} -1)} (\{t\}) = 1$.
Between the shift times, the number of particles at a site $j$ grows according to $N^{(\tip{\eta_{t-}} + j)}$ for $(\zeta _t)$, that is, $\zeta _t (  j) - \zeta _{t- } (  j) = 1$ if and only if $N^{(\tip{\eta_{t-}} +j)}_{t } - N^{(\tip{\eta_{t-}} + j)}_{t- } = 1$.

Let us now list some of the properties of the processes $(\zeta _t)$ and $(\xi _t)$
which are used later on. They follow from definitions and construction
of $(\zeta _t)$ and $(\xi _t)$.
\begin{enumerate}
	\item A.s.  for all $t \geq 0$, $\xi _t (0) \geq 1$ and $\zeta _t (0) \geq 1$.
	\item Every shift for $(\zeta _t)$ is a shift for  $(\xi _t)$ too,
	since for $m = 1$,
	$$\sum\limits _{k \geq 0} \frac{\xi _{t-}(k)}{(k + m)^{2\alpha }} \geq \xi_{t-}(0)\geq 1.$$
	\item If a shift occurs for $(\zeta _t)$ ($(\xi _t)$) at time $t$, then $\zeta _t(0) = 1$ ($\xi _t(0) = 1$ respectively).
	\item  If 
	$(\xi _t(j))$ is increased 
	by $1$ at time $t$, $j \in \Z _+$,
	then so is $(\zeta _t (j))$
	(but not necessarily vice versa by  \eqref{sass}).
	\item The processes $(\eta _t)$, 
	$(\xi _t)$, $(\zeta _t)$
	are Markov processes with respect 
	to $\{ \mathscr{F} _t, t\geq 0 \}$.
\end{enumerate}

Let $(\varphi _k)_{k \in \N}$ be the shift times of $(\xi _t)$, that is, $t \in \{ \varphi _k\}_{k \in \N}$ if and only if for some $m \in \N$
$$
\mathbf{N}(\{t\} \times \{\tip {\eta _t}- m \} \times [0, \sum\limits _{k \geq 0} \frac{\xi _{t-}(k)}{(k + m)^{2\alpha }} ]) = 1,
$$
or alternatively if for some $m \in \N$
\[
N^{(\tip{\eta_t}-m)}_{t} - N^{(\tip{\eta_t}-m)}_{t-}  =1 \qquad \text{and} \qquad u^{(\tip {\eta _t} -m)}_{N^{(\tip {\eta _t} -m)} _{t} } \leq \sum\limits _{k \geq 0} \frac{\xi _{t-}(k)}{(k + m)^{2\alpha }}.
\]

Denote by $(\tilde N _t)$ the Poisson process 
such that $\tilde N _{t} - \tilde N _{t-} = 1$
for those $t$ when  
$ N ^{\tip{\eta _t } -1 } _{t} - 
N ^{\tip{\eta _t } -1 } _{t-}= 1$,
so that
$(\tilde N _t)$ is the Poisson process 
whose jumps
are exactly the shift times 
for $( \zeta _t )$.
Let
$\sigma _k = \inf\{t >0: \tilde N_t = k \}$
be the jump times of the process
$(\tilde N_t )$,
that is, $t \in \{ \sigma _k\}_{k \in \N}$
if and only if $\tilde N _{t} - \tilde N _{t-} = 1$.
Let also $\varphi _k = \sigma _k = 0$ for $k = 0,-1,-2,...$
Note that $\{ \sigma _k\}_{k \in \N} \subset \{ \varphi _k\}_{k \in \N} $ since every shift for $(\zeta _t)$ is a shift for $(\xi _t)$ too.
The process $(\zeta _t)$ has the following representation (let us stress here that 
we do not use this representation in the proofs): for $t\geq 0$ 
let $n\in \N$ be such that  $t \in [\varphi _n, \varphi _{n+1})$, then 
\begin{equation*}
	\zeta  _t (j) = 1 + \sum\limits_{\substack{k \in \{0,1,...,n\}: \\
			 \tilde N_{\varphi _k} +j \geq \tilde N_{\varphi _n}}}^n 
	N^{\left(\tip {\eta _{\varphi _k}} + j +  \tilde N_{\varphi _k} - \tilde N_{\varphi _n} \right)}(\varphi _k, \varphi _{k+1}\wedge t], \quad j \in \Z _+.
\end{equation*}

\begin{prop}  \label{astute}
 $(\zeta _t)$ stochastically dominates $(\xi _t)$.
 \end{prop}
 \begin{proof}
Let us show that \eqref{permeate}
is satisfied for every $k = 0,1,...$
if we take $\hat \xi _t = \xi _t$
and $\hat \zeta _t = \zeta _t$.

We use induction on $k$. For $k = 0$
\eqref{permeate} is clear since 
by construction 
every shift 
of $(\zeta _t)_{t\geq 0}$
is a shift for 
$(\xi _t )_{t\geq 0}$ too, 
while 
every time $(\xi _t (0))_{t\geq 0}$
is increased  by $1$
$(\zeta _t (0))_{t\geq 0}$ is increased too.

Fix $n \in \N $  and assume that \eqref{permeate} holds for   $k=0, \ldots, n-1$.
At $t = 0$ \eqref{permeate} with $k=n$ holds.
Let $\theta < \infty $ be the first moment when \eqref{permeate} with $k=n$ does not hold;
note that $\theta $ is well defined since a.s. there are only finitely many shifts up to any time moment, 
and finitely many increments at sites $0, 1, \dots, n$ took place.
Thus we have 
\begin{equation} \label{opulent}
  \sum\limits _{i = 0} ^n  \xi _{\theta - }(i) \leq 
  \sum\limits _{i = 0} ^n  \zeta _{\theta - }(i)
\end{equation}
but 
\begin{equation} \label{lush}
  \sum\limits _{i = 0} ^n  \xi _{\theta  }(i) > 
  \sum\limits _{i = 0} ^n  \zeta _{\theta  }(i).
\end{equation}

If $(\xi _t)$ got shifted by $m$ at $\theta$,
then, at $\theta$, $(\zeta _t)$ got shifted by $1$
or did not change; in either case
\[
 \sum\limits _{i = 0} ^n  \xi _{\theta  }(i)
 \leq  
 1 + \sum\limits _{i = 0} ^{n-1} \xi _{\theta - }(i)
 \leq 
  1 + \sum\limits _{i = 0} ^{n-1} \zeta _{\theta - }(i)
  \leq
   \sum\limits _{i = 0} ^n  \zeta _{\theta  }(i).
\]

If on the other hand $(\xi _t)$
got increased by $1$ 
at a site
 $j$, $0 \leq j \leq n$,
at $\theta$,
then $(\zeta _t)$ got increased 
by $1$ at the same  site too.
So, \eqref{opulent}
 and \eqref{lush}
 cannot both be satisfied for a finite $\theta$,
 and thus we have a contradiction.
 \end{proof}

We now introduce another $\Z _+ ^{\Z _+ }$-valued process defined by
\begin{equation}\label{long-winded}
 \bar \zeta _t (k) = 1 +
 N^{(n-k)}(\sigma _{n-k},t], \quad
 t \in (\sigma _n, \sigma _{n+1}],
\end{equation}
which is equal in distribution to $(\zeta _t)$
by the strong Markov property of a Poisson point process,
see the appendix in \cite{shapenodeath}.
It is a little bit easier to work with, so we will use it in the estimates below.

Denote the distance from the leftmost occupied site for $(\eta _t)$
to the origin by $X_t$, so that
\begin{equation*}
	X_t := -\tip{\eta_t}.
\end{equation*}
Note that $(X_t)$ allows the representation
\begin{equation}\label{tawdry}
X_t = \sum\limits _{m \in \N} m \int\limits _{(0,t] \times [0,1]} \1_{[0,b^{(d)}(\tip{\eta _{s-}} - m, \eta _{s-})]}(u) N^{(\tip{\eta _{s-}} - m)}(ds  du).
 \ \ \ t \geq 0
\end{equation}
To represent $X_t$ as an integral with respect to a Poisson point process, for $0<a<b$ and $m \in \N$
 define the set
$T(a,b,m) = \set{(s,k) \in \R_+ \times \Z}{ a < s\leq b, \tip{\eta _{s-}} + m = k  }$ and the point process
\begin{equation}\label{chuck}
N^{(X)}((a,b]\times\{m\}\times U) = \mathbf{N}(T(a,b,m)\times U),  \ \ \
0<a<b, \ m \in \N, \ U \in \mathscr{B}[0,1].
\end{equation}
Note that for $0<a<b$ a.s. 
\[
N^{(X)}((a,b]\times\{m\}\times U)\1 \{\tip{\eta _{a}} = \tip{\eta _{b}} \} = 
\mathbf{N}((a,b]\times\{ \tip{\eta _{a}}+m\}\times U)\1 \{\tip{\eta _{a}} = \tip{\eta _{b} }\}
\]

It follows from the strong Markov property for 
a Poisson point process
 (as formulated in the appendix in \cite{shapenodeath}) that $N^{(X)}$ 
is a Poisson point process; also, $N^{(X)}$ is equal in distribution to $\mathbf{N}$.
It follows from \eqref{tawdry} and \eqref{chuck} that
\begin{equation}\label{a presto sim a dopo}
X_t = \int\limits _{(0,t]\times \N \times [0,1]}
m
\1_{[0,b^{(d)}(\tip{\eta _{s-}} - m, \eta _{s-})]}(u)
N^{(X)}(ds dm du).
\end{equation}

The process
\begin{align}\label{racy}
  M_t &:= X_t - \int\limits _{0}^t \sum\limits _{m \in \N} m b^{(d)}(\tip{\eta _{s-}} - m, \eta _{s-}) ds \nonumber 
  \\ &= X_t - \int\limits _{0}^t\sum\limits _{m \in \N} m \left( 1 \wedge \sum\limits _{k = 0} ^ \infty \frac{\xi _{s-}(k)}{(m + k )^{2\alpha }} \right) ds, \quad t \geq 0, 
  \end{align}
is therefore a local martingale with respect to 
$\{ \mathscr{F} _t, t\geq 0 \}$, see e.g.
(3.8) in
Section 3, Chapter 2 in
\cite{IkedaWat}.
We will see in Lemma \ref{semblance} 
below that $(M_t)$ is 
a (true) martingale.
We denote by $Q_t$ the second summand on
the right hand side of \eqref{racy},
so that 
\begin{align}\label{racy2}
M_t = X_t - Q_t.
  \end{align}

In the remaining part of this section we prove that 
$(X_t)$ grows at most linearly (Theorem \ref{at most linear}).
First we prove that $(Q_t)$ grows at most linearly (Proposition \ref{dibber}),
then we show that the martingale $(M_t)$ has some nice properties (Proposition \ref{fluty})
which allow  us to apply a strong law of large numbers for martingales in the proof of 
Theorem \ref{at most linear}.
The following lemma collects some relatively straightforward properties which are used multiple times 
in the rest of this section.
\begin{lem}\label{covariance calc}
	Let $\beta$, $X$ and $Y$ be non-negative random variables with finite third moment.
	\begin{enumerate}[label={(\roman*)}]
  	\item if $\beta \perp (X,Y)$ ($\beta$ is independent to $(X,Y)$), then 
			\[
    		\mathrm{Cov}(\beta X, Y) = \E \beta \mathrm{Cov} ( X, Y);
   		\]
		\item if $X \mid \beta \perp Y\mid \beta$ (that is, $X$ and $Y$ are conditionally independent given $\beta$) and 
		${\Ec{X}{\beta} = \Ec{Y}{\beta} = \beta}$, then 
			\[
				\mathrm{Cov}( X, Y) = \mathrm{Var}(\beta);
   		\]
  	\item if $\E(X\mid \beta) = \beta$, then 
    	\[
    		\mathrm{Cov}( X, \beta) = \mathrm{Var}(\beta),  \ \ \  E\beta X = E \beta ^2;
   		\]
    \item if $\E(X\mid \beta) = \E(Y\mid \beta) = \beta$ and $X\mid \beta \perp Y\mid \beta$, then 
    	\[
    		\mathrm{Cov}( \beta X,Y ) = \E \beta ^3 - \E \beta ^2 \E \beta;
   		\]
   	\item if $\E(X\mid \beta) =  \beta$, $\E(X^2 \mid \beta) =  \beta^2 + \beta$ and $Y \perp (X,\beta)$, then
    	\[
    		\mathrm{Cov}( \beta (X+Y),(X+Y) ) = \E \beta ^2 +\E \beta ^3 - \E \beta ^2 \E \beta + \E Y  \mathrm{Var}(\beta) + \E \beta \mathrm{Var}(Y);
   		\]
   	\item if $N$ is a Poisson process independent of $\beta$, then 
			\[
    		\Ec{N(\beta)}{\beta} = \beta.
   		\]
 \end{enumerate}

 \end{lem}
 
 \begin{proof} The proof is based on the properties of conditional expectation.
 	The proofs of $(i)$-$(v)$ are done by conditioning on $\beta$.
 	  We give the proofs for $(ii)$, $(iv)$ and $(vi)$ only; the others are similar to $(ii)$
 	  and $(iv)$.
 	  For $(ii)$,
 	  \begin{multline*}
 	  \mathrm{Cov}( X, Y) = \E XY - \E X \E Y = \E \Ec{XY}{\beta} - \left(\E{\beta} \right) ^2
 	  \\
 	   = 
 	   \E \left( \Ec{X}{\beta} \Ec{Y}{\beta}\right) - \left(\E{\beta} \right) ^2 = 
 	   \E \beta ^2 -\left(\E{\beta} \right) ^2 = \mathrm{Var}(\beta).
 	  \end{multline*}
 For $(iv)$,
 \begin{multline*}
  \mathrm{Cov}( \beta X,Y ) = \E \beta X Y - \E \beta X  \E  Y = \E \Ec{\beta X Y}{\beta} - \E  \beta  \E \Ec{\beta X }{\beta}\\ 
 = \E\beta \Ec{ X Y}{\beta} - \E  \beta  \E \big( \beta  \Ec{X }{\beta} \big) = \E\big( \beta \Ec{ X }{\beta} \Ec{ Y}{\beta}\big) -  \E  \beta  \E  \beta ^2 = \E  \beta ^3 - E  \beta  \E  \beta ^2.
 \end{multline*}

To prove $(vi)$
we use the disintegration theorem for regular conditional
probability distribution, see e.g. Kallenberg \cite[Theorem 6.4]{KallenbergFound}.
To adapt to the notation in the preceding reference, 
let $S = D([0,+\infty), \R )$ (the Skorokhod space)
equipped 
with the cylindrical $\sigma$-algebra,
and $T = \R _+$ equipped with the Borel $\sigma$-algebra,
and consider $N$ and $\beta$ as random elements 
in $S$ and $T$ respectively. Note that since 
$N$ and $\beta$ are independent, 
the regular conditional probability distribution 
of $N$ given $\beta$ is simply the distribution of $N$ in $S$,
which we denote by $\nu$.
Define $f(s,t) = s(t)$, $s \in S$, $t \in T$.
For every $q \geq 0$,
\[
 \int\limits _{S} \nu(ds) s(q) = \E N(q) = q,
\]
hence by the disintegration theorem a.s.
\[
 \Ec{f(N,\beta)}{\beta} = 
 \int\limits _{S} \nu(ds) f(s, \beta) =
 \int\limits _{S} \nu(ds) s(\beta) = \beta.
\]
 \end{proof}
\begin{rmk}
 Concerning item $(vi)$, note that the conditional distribution
 of $N(\beta)$ given $\beta$ is 
 $\mathcal{P}ois(\beta)$, 
 where
 $\mathcal{P}ois(q)$
 is the Poisson distribution 
 with parameter 
 $q \geq 0$.
\end{rmk}

\begin{lem}\label{semblance}
 The process $(M_t)$ is a true martingale.
\end{lem}
\begin{proof}
By Lemma \ref{covariance calc} for every $t \geq 0$,
\begin{align*}
 \E{Q _t} &= \E \int\limits _{0}^t\sum\limits _{m \in \N} m \left( 1 \wedge \sum\limits _{k = 0} ^ \infty \frac{\xi _{s-}(k)}{(m + k )^{2\alpha }} \right) ds \leq \E \int\limits _{0}^t\sum\limits _{m \in \N} m  \sum\limits _{k = 0} ^ \infty \frac{\bar \zeta _{s-}(k)}{(m + k )^{2\alpha }} ds \\ 
 	&= \int\limits _{0}^t\sum\limits _{m \in \N} m \sum\limits _{k = 0} ^ \infty \frac{1 + \E N^{(n-k)}(\sigma _{n-k},s] \1\{\sigma _{n-k} \leq t \}}{(m + k )^{2\alpha }} ds \\
  &= \int\limits _{0}^t\sum\limits _{m \in \N} m \sum\limits _{k = 0} ^ \infty \frac{1 + \E (s  - \sigma _{n-k})_+ }{(m + k )^{2\alpha }} ds \leq \int\limits _{0}^t\sum\limits _{m \in \N} m \sum\limits _{k = 0} ^ \infty \frac{1 + s }{(m + k )^{2\alpha }} ds \\ 
  &< t(t+1) \sum\limits _{m \in \N} \frac{m}{m ^{\alpha }} \sum\limits _{k = 0} ^ \infty \frac{1}{ k ^{\alpha}} = t(t+1) \sum\limits _{m \in \N} \frac{1}{m ^{\alpha -1 }} \sum\limits _{k = 0} ^ \infty \frac{1}{ k ^{\alpha}}.
 \end{align*}
 and hence for every $t \geq 0$
 \[
 \E \sup\limits _{s \leq t} |M_s| \leq \E|X_t| + \E|Q_t| = 2 \E|Q_t| < \infty. 
 \]

The statement of the lemma now follows from Theorem 51 in Protter \cite{Pro05}.
\end{proof}

 	The following proposition is a key step in the proof of the main result of this section, Theorem \ref{at most linear}.
 	We establish here that $(Q_t) $ grows at most linearly with $t$.
 	
 \begin{prop} \label{dibber}
 	
 		$(i)$ There exists $C >0$ such that 
 		a.s. for sufficiently large t,
 		\begin{equation}
 		\int\limits _{s = 0} ^t ds
 		\sum\limits _{k \in \Z _+} 
 		\frac{ \xi _s (k)}{k^\alpha} < Ct.
 		\end{equation}
 		
 		$(ii)$ 	There exists $C >0$ such that 
 		a.s. for sufficiently large t,
 		\begin{equation*}
 		Q_t \leq Ct.
 		\end{equation*}

 \end{prop}

\begin{proof}
First we show that $(i)$ implies $(ii)$.
 Indeed,
\begin{equation*}
\begin{gathered}
Q_t = \int\limits _{0}^t\sum\limits _{m \in \N} 
m \left( 1 \wedge \sum\limits _{k = 0} ^ \infty
\frac{\xi _{s-}(k)}{(m + k )^{2\alpha }}
\right) ds
\leq 
\int\limits _{0}^t\sum\limits _{m \in \N} 
m \sum\limits _{k = 0} ^ \infty
\frac{\xi _{s-}(k)}{m^{\alpha}k ^{\alpha}} ds
= \sum\limits _{m \in \N} 
\frac{1}{m^{\alpha - 1}}
\int\limits _{0}^t  \sum\limits _{k = 0} ^ \infty
\frac{\xi _{s-}(k)}{k ^{\alpha}} ds,
\end{gathered}
\end{equation*}
so that $(i)$ yields  $(ii)$.

The rest  is devoted to the proof of $(i)$.
		 By Lemma
 \ref{rambunctious kitty}
 and Proposition \ref{astute},
 
 \begin{equation}\label{twinge}
  \int\limits _{s = 0} ^t ds
  \sum\limits _{k \in \Z _+} 
 \frac{1}{k^\alpha} \xi _s (k) \leq 
  \int\limits _{s = 0} ^t ds
  \sum\limits _{k \in \Z _+} 
 \frac{1}{k^\alpha} \zeta _s (k).
 \end{equation}
 
 Define $\sigma(-i) = 0$, $i \in \N$, and
\begin{equation}\label{def Y_n}
 Y_ n = \sum\limits _{k \in \Z _+}
 \frac{1+ N^{(n-k)}(\sigma _{n-k},\sigma _{n+1}]}{k^\alpha} .
\end{equation}

Recall that the process $(\bar \zeta _t)$ was defined in \eqref{long-winded}.
Clearly
\begin{equation}\label{hose}
 Y_ n \geq \sum\limits _{k \in \Z _+} 
 \frac{1}{k^\alpha} \bar \zeta _t (k),
 \quad
 t \in (\sigma _n, \sigma _{n+1}].
\end{equation}

Combining \eqref{twinge}
and \eqref{hose} and recalling that
$(\zeta _t) \overset{d}{=} (\bar \zeta _t) $
result in the observation that 
it is sufficient to show that
the strong law of large numbers holds for
$(Z_n)_{n \in \N}$, where 
\[
 Z_n:= (\sigma_{n+1} - \sigma _n) Y_n.
\]

As jump times of a Poisson process, $\sigma_{n+1} - \sigma _n$ are independent unit exponentials,
in particular 
\[
 \EE{(\sigma_{n+1} - \sigma _n)^k} = k!,  \ \ \ k \in \N.
\]
Note that for every $n \in \N$
\begin{align}\label{alpha is three sedonds}
 \E Z_n &= \E \Big[(\sigma_{n+1} - \sigma_n) \sum\limits_{k \in \Z _+} \frac{1}{k^\alpha} N^{(n-k)}(\sigma _{n-k},\sigma _{n+1}] \Big]
		+ \sum\limits _{k \in \Z _+} \frac{1}{k^\alpha} \nonumber \\
		&= \sum\limits _{k \in \Z _+} \frac{1}{k^\alpha} \E \left[  (\sigma_{n+1} - \sigma _n) N^{(n-k)}(\sigma _{n-k},\sigma _{n}] \right] + \sum\limits _{k \in \Z _+} \frac{1}{k^\alpha} \E \left[ (\sigma_{n+1} - \sigma _n) N^{(n-k)}(\sigma _n,\sigma _{n+1}] \right] \nonumber \\ 
		&\quad + \sum\limits _{k \in \Z _+} \frac{1}{k^\alpha} =  \sum\limits _{k \in \Z _+} \frac{k}{k^\alpha} + 3\sum\limits _{k \in \Z _+} \frac{1}{k^\alpha},
\end{align}
and the last two sums are finite. Thus $\E Z_n$
is bounded in $n$.
In \eqref{alpha is three sedonds} we applied Lemma \ref{covariance calc} $(iii)$.
In this proof
 we make use of 
Lemma \ref{covariance calc}
in multiple places.

The random variables $\{Z_n \}_{n \in \N}$ are not independent, however the covariance is small for distant elements: we are going to show that there exists a constant $C_{_Z} >0$ such that for $n, m \in \N$.
\begin{equation}\label{rebut}
	\mathrm{Cov}(Z_n, Z_{n+m}) \leq \frac{C_{_Z}}{m^{\alpha - 1}}.
\end{equation}

We have
\begin{align}\label{stud}
  \mathrm{Cov}(Z_n, Z_{n+m}) &= \mathrm{Cov} \Big( \sum\limits _{i \in \Z _+} \frac{\sigma _{n+1} - \sigma _n}{i^\alpha} N^{(n-i)}(\sigma _{n-i},\sigma _{n+1}], \nonumber \\ 
	&\qquad\qquad \sum\limits _{j \in \Z _+} \frac{\sigma _{n+m+1} - \sigma _{n+m}}{j^\alpha} N^{(n+m-j)}(\sigma _{n+m-j},\sigma _{n+m+1}] \Big) \nonumber \\
 &= \sum\limits _{i,j \in \Z _+} \frac{1}{i^\alpha j^\alpha} \mathrm{Cov}\Big( (\sigma _{n+1} - \sigma _n)N^{(n-i)}(\sigma _{n-i},\sigma _{n+1}],  \nonumber  \\
 &\qquad\qquad (\sigma _{n+m+1} - \sigma _{n+m})N^{(n+m-j)}(\sigma _{n+m-j},\sigma _{n+m+1}] \Big).
\end{align}

Let us denote by $\cov(i,j)$ the covariance in the last sum of \eqref{stud}.  
Recall that we defined $\sigma _k = 0$ for $k = 0, -1, -2, ...$.
We can split the interval $(\sigma _{n+m-j},\sigma _{n+m+1}]$ as follows:
\begin{multline*} 
(\sigma _{n+m-j},\sigma _{n+m+1}]  \\ =
	\begin{cases}
			(\sigma _{n+m-j},\sigma _{n-i}] \cup (\sigma _{ n-i},\sigma _n]
		 	\cup (\sigma _n,\sigma _{n+1}] 
			\cup (\sigma_{ n+1 },\sigma _{n+m+1}],
			 & \text{if } j > m +i 
		\\
		(\sigma _{n+m-j},\sigma _n]
		\cup (\sigma _n,\sigma _{n+1}] \cup (\sigma_{ n+1 }, \sigma _{n+m+1}], \
		&  \text{if } m + i \geq j > m 
		\\
		(\sigma _n,\sigma _{n+1}] \cup (\sigma_{n+1 },\sigma _{n+m+1}], \
	    &	\text{if }  j = m 
	    \\ (\sigma_{n+m-j },\sigma _{n+m+1}], 
	    &	\text{if }  j < m 
	\end{cases}
\end{multline*}
or, alternatively, 
\begin{multline}
(\sigma _{n+m-j},\sigma _{n+m+1}]  \\ =
\begin{cases}
\begin{aligned}
(\sigma _{n+m-j},\sigma _{n-i}] \cup (\sigma _{(n+m-j)\vee (n-i)},\sigma _n]
\cup (\sigma _n,\sigma _{n+1}] 
\\
\cup (\sigma_{(n+m-j)\vee (n+1) },\sigma _{n+m+1}],
\end{aligned}
& \text{if } j > m +i 
\\
(\sigma _{(n+m-j)\vee (n-i)},\sigma _n]
\cup (\sigma _n,\sigma _{n+1}] \cup (\sigma_{(n+m-j)\vee (n+1) }, \sigma _{n+m+1}], \
&  \text{if } m + i \geq j > m 
\\
		(\sigma _n,\sigma _{n+1}] \cup (\sigma_{(n+m-j)\vee (n+1) },\sigma _{n+m+1}], \
&	\text{if }  j = m 
\\ (\sigma_{(n+m-j)\vee (n+1) },\sigma _{n+m+1}], 
&	\text{if }  j < m,
\end{cases}
\end{multline}
and hence (with convention that $(a,b] = \varnothing$ if $a> b$)
\begin{equation}  \label{rancid}
\begin{split}
(\sigma _{n+m-j},\sigma _{n-i}] \ne \varnothing \text{ and }
(\sigma _{n+m-j},\sigma _{n-i}] \subset (\sigma _{n+m-j},\sigma _{n+m+1}] & \Leftrightarrow j > m +i, \\
(\sigma _{(n+m-j)\vee (n-i)},\sigma _n]  \ne \varnothing \text{ and }
(\sigma _{(n+m-j)\vee (n-i)},\sigma _n] \subset (\sigma _{n+m-j},\sigma _{n+m+1}] & \Leftrightarrow j > m, \\
(\sigma _n,\sigma _{n+1}] \subset (\sigma _{n+m-j},\sigma _{n+m+1}] & \Leftrightarrow j \geq m. \\
\end{split}
\end{equation}
We now proceed to estimate $\cov(i,j)$. Using \eqref{rancid} we get 
\begin{align}\label{grovel}
\cov(i,j) &= \mathrm{Cov} \Big( (\sigma _{n+1} - \sigma _n) \left\{ N^{(n-i)}(\sigma _{n-i},\sigma _n]+ N^{(n-i)}(\sigma _n,\sigma _{n+1}] \right\}, \nonumber \\
	&\qquad\qquad (\sigma _{n+m+1} - \sigma _{n+m}) \Big\{ \1\{j>m+i \} N^{(n+m-j)}(\sigma _{n+m-j},\sigma _{n-i}] \nonumber \\
 	&\qquad\qquad\quad+ \1\{j>m\} N^{(n+m-j)}(\sigma _{(n+m-j)\vee (n-i)},\sigma _n] 
  + \1\{j \geq m\} N^{(n+m-j)}(\sigma _n,\sigma _{n+1}] \nonumber \\
  &\qquad\qquad\quad+ N^{(n+m-j) }(\sigma_{(n+m-j)\vee (n+1) },\sigma _{n+m+1}] \Big\} \Big) \nonumber \\
 	&= s_{11}+ s_{12} + s_{13}+s_{14} +s_{21}+ s_{22} + s_{23}+s_{24},
\end{align}
where $s_{uv}$, $u\in \{1,2 \}$, $v \in \{1,2,3,4 \}$, stands for the covariance of $u$-th and $v$-th summands in the decomposition in \eqref{grovel}, for example
\begin{multline*}
 s_{23} = \mathrm{Cov} \Big( (\sigma _{n+1} - \sigma _n) N^{(n-i)}(\sigma _n,\sigma _{n+1}], 
 (\sigma _{n+m+1} - \sigma _{n+m}) \1\{j \geq m\} N^{(n+m-j)}(\sigma _n,\sigma _{n+1}] \Big).
\end{multline*}

Let us estimate each of $s_{uv}$. To start off, $s_{11} = s_{21} =s_{14} = s_{24} = s_{22} =0$ as the covariance of independent random variables.
In particular,
\begin{multline*}
 s_{22} = \1\{j>m\} \mathrm{Cov} \Big( (\sigma _{n+1} - \sigma _n) N^{(n-i)}(\sigma _n,\sigma _{n+1}],\\ 
 (\sigma _{n+m+1} - \sigma _{n+m}) N^{(n+m-j)}(\sigma _{(n+m-j)\vee (n-i)},\sigma _n] \Big) = 0.
\end{multline*}

To other 
terms we apply Lemma \ref{covariance calc}.
 Assume first that $n-i \ne n+m-j$.
We have by Lemma \ref{covariance calc} $(i)$, $(ii)$,
and $(vi)$,
\begin{align*}
 s_{12} &= \1\{j>m\} \mathrm{Cov} \Big( (\sigma _{n+1} - \sigma _n) N^{(n-i)}(\sigma _{n-i},\sigma _n], \\
 	&\qquad (\sigma _{n+m+1} - \sigma _{n+m}) N^{(n+m-j)}(\sigma _{(n+m-j)\vee (n-i)},\sigma _n] \Big) \\
 	&= \1\{j>m\} \E \left(\sigma _{n+1} - \sigma _n\right) \E \left(\sigma _{n+m+1} - \sigma _{n+m}\right) \\
  &\qquad \times \mathrm{Cov} \left(N^{(n-i)}(\sigma _{n-i},\sigma _n], N^{(n+m-j)}(\sigma _{(n+m-j)\vee (n-i)},\sigma _n] \right) \\
  &\leq \1\{j>m\} \mathrm{Cov} \left(N^{(n-i)}(\sigma _{n-i},\sigma _n], N^{(n+m-j)}(\sigma_{n-i} ,\sigma _n] \right) \\
 &=  \1\{j>m\} \mathrm{Var}\big( \sigma _n - \sigma _{n-i} \big) = i \1\{j>m\}.
\end{align*}
Applying Lemma \ref{covariance calc}
$(iii)$, we continue
\begin{align*}
 s_{13} &= \1\{j\geq m\} \mathrm{Cov} \Big( (\sigma _{n+1} - \sigma _n) N^{(n-i)}(\sigma _{n-i},\sigma _n], 
 	(\sigma _{n+m+1} - \sigma _{n+m}) N^{(n+m-j)}(\sigma _n,\sigma _{n+1}] \Big) \\
 	&= \1\{j\geq m\} \E(\sigma _{n+m+1} - \sigma _{n+m}) \E N^{(n-i)}(\sigma _{n-i},\sigma _n]  \mathrm{Cov} \Big( (\sigma _{n+1} - \sigma _n), N^{(n+m-j)}(\sigma _n,\sigma _{n+1} \Big) \\
 	&= \1\{j\geq m\} i \mathrm{Var} \big( \sigma_{n+1} - \sigma _n \big) = \1\{j\geq m\} i.
\end{align*}

In the same spirit by Lemma \ref{covariance calc} $(iv)$
\begin{align*}
 s_{23} &= \1\{j \geq m\} \mathrm{Cov} \Big( (\sigma _{n+1} - \sigma _n) N^{(n-i)}(\sigma _n,\sigma _{n+1}],  (\sigma _{n+m+1} - \sigma _{n+m}) N^{(n+m-j)}(\sigma _n,\sigma _{n+1}] \Big) \\
 &= \1\{j \geq m\} \E (\sigma _{n+m+1} - \sigma _{n+m}) \mathrm{Cov} \Big( (\sigma _{n+1} - \sigma _n) N^{(n-i)}(\sigma _n,\sigma _{n+1}], N^{(n+m-j)}(\sigma _n,\sigma _{n+1}] \Big) \\
 &= \1\{j \geq m\} \left[ \E (\sigma _{n+1} - \sigma _n)^3 - \E (\sigma _{n+1} - \sigma _n)^2 \right] = (3! - 2!) \1\{j \geq m\} = 4 \1\{j \geq m\}.
\end{align*}
The computations started from \eqref{grovel} imply that 
\begin{equation}\label{nitty-gritty}
	\cov(i,j) \leq (2i +4 ) \1\{j \geq m\}
\end{equation}
provided $j \ne m+i$.
 
If $j = m+i$ then
\begin{align}\label{surreptitious}
	\cov(i,j) &= \mathrm{Cov}\Big( (\sigma _{n+1} - \sigma _n)N^{(n-i)}(\sigma _{n-i},\sigma _{n+1}], (\sigma _{n+m+1} - \sigma _{n+m})N^{(n-i)}(\sigma _{n-i},\sigma _{n+m+1}] \Big) \notag \\
		&= \mathrm{Cov}\Big( (\sigma _{n+1} - \sigma _n)N^{(n-i)}(\sigma _{n-i},\sigma _{n+1}], (\sigma _{n+m+1} - \sigma _{n+m})N^{(n-i)}(\sigma _{n-i},\sigma _{n+1}] \Big) + 0  \notag \\
 		&= \mathrm{Cov}\Big( (\sigma _{n+1} - \sigma _n)N^{(n-i)}(\sigma _{n-i},\sigma _{n+1}], N^{(n-i)}(\sigma _{n-i},\sigma _{n+1}] \Big)  \notag \\
 		&= \mathrm{Cov}\Big( (\sigma _{n+1} - \sigma _n) \Big\{ N^{(n-i)}(\sigma _{n-i},\sigma _n] + N^{(n-i)}(\sigma _n,\sigma _{n+1}] \Big\},  \notag \\ 
	&\qquad\qquad \Big\{ N^{(n-i)}(\sigma _{n-i},\sigma _n] + N^{(n-i)}(\sigma _n,\sigma _{n+1}] \Big\} \Big) \notag \\
  &= 2 + 6 - 2 + i + 2i = 3i + 6
\end{align}
 by Lemma \ref{covariance calc}
 $(v)$ where we can take $\beta = \sigma _{n+1} - \sigma _n$,
 $X = N^{(n-i)}(\sigma _n,\sigma _{n+1}]$ and $Y = N^{(n-i)}(\sigma _{n-i},\sigma _n]$.
  Note that
 $\E N^{(n-i)}(\sigma _{n-i},\sigma _n] = 
 \E (\sigma _n - \sigma _{n-i}) = i$,
 \[
 \Ec{ \left( N^{(n-i)}(\sigma _n,\sigma _{n+1}]\right) ^2}{(\sigma _n,\sigma _{n+1}]}
= \left( \sigma _{n+1} - \sigma _n \right) ^2
+
 \sigma _{n+1} - \sigma _n,
 \]
and 
\[
\mathrm{Var}(N^{(n-i)}(\sigma _{n-i},\sigma _n]) = \Ec{\left(N^{(n-i)}(\sigma _{n-i},\sigma _n]\right) ^2}{\sigma _n - \sigma_{ n-i}}  - 
\left( \EE{N^{(n-i)}(\sigma _{n-i},\sigma _n]} \right) ^2
\]
\[
 = \EE{ \left(\sigma _n - \sigma_{ n-i} \right) ^2 + \sigma _n - \sigma_{ n-i}} 
-
\left( \EE{\sigma _n - \sigma_{ n-i}} \right) ^2 = 
\EE{  \sigma _n - \sigma_{ n-i}} + \mathrm{Var}(\sigma _n - \sigma_{ n-i}) = 2i.
\]

 In conjunction with \eqref{nitty-gritty},
 \eqref{surreptitious} allows us to estimate 
  $\mathrm{Cov}(Z_n, Z_{n+m})$.
  Recalling \eqref{stud}, we get
  \begin{equation}\label{surly}
   \begin{gathered}
    \mathrm{Cov}(Z_n, Z_{n+m}) 
    = \sum\limits _{i,j \in \N }
 \frac{1}{i^\alpha j^\alpha} 
 \cov(i,j)  \leq  \sum\limits _{i,j \in \N , i \ne j}
 \frac{(2i +4 ) \1\{j \geq m\}}{i^\alpha j^\alpha} 
 + \sum\limits _{i \in \N}
 \frac{3i + 6}{i^\alpha (i+m)^\alpha}
 \\
 \leq 
 \sum\limits _{i \in \N}
 \frac{2i +4  }{i^\alpha } 
 \sum\limits _{ j = m} ^{\infty}
 \frac{1}{ j^\alpha} + 
  \frac{1}{m^\alpha}
 \sum\limits _{i \in \N}
 \frac{3i + 6}{i^\alpha} .
   \end{gathered}
  \end{equation}
Since $\sum\limits _{ j = m} ^{\infty}
 \frac{1}{ j^\alpha} = O(\frac{1}{m^{\alpha - 1}})$ as $m \to \infty$,
 \eqref{surly} implies \eqref{rebut}.
 The statement of $(i)$ follows
 from \eqref{rebut}
 and
 the strong law of large numbers 
for dependent random variables,
see e.g.
Hu,  Rosalsky, and Volodin \cite{HRV08},
or Corollary 11 of
Lyons \cite{L88}.

 \end{proof}

 Let $\Delta M _n = M_{n+1} - M_n$, 
 and let $\Delta X _n $ and $\Delta Q _n $
 be defined in the same way.
 In the following proposition we establish finiteness
 of a moment of the martingale difference $\Delta M _n$. Later on this allows 
us to apply a strong law of large numbers for martingales to $M_n$. 
 
 \begin{prop} \label{fluty}
  Let $p \in (1,  (\alpha  - 1)) \cap (1,2]$. Then
   $ \E \left| \Delta M _n \right| ^p $
   is bounded uniformly in $n$.
 \end{prop}
 \begin{proof}
 By \eqref{a presto sim a dopo},
 \begin{equation}
  \Delta X _n = \int\limits _{s \in (n,n+1], m \in \N, u \in \R _+}
  m \1 \Big\{ u \leq 1\wedge
   \sum\limits _{k \in \Z _+} 
 \frac{ \xi _{s-} (k)}{(k+m)^{2\alpha} }
 \Big\} N ^{(X)}(ds dm du)
 \end{equation}
 Note that for every $k \in \Z$ and $s \geq 0$, $\E \eta _{s} (k) \leq k+1$ because $(\eta _{t} (k) - \eta _{0} (k))_{t\geq 0}$ is dominated by a Poisson process, and consequently also 
 \[
	 \E \xi _{s-} (k)  \leq k +1.
 \]

 Novikov's inequalities  for discontinuous martingales (also known as `Bichteler-Jacod's inequalities'; see Novikov \cite{Nov75}, or Marinelli and R\"ockner \cite{MR14} for generalizations and historical discussions) give 
 \begin{align}\label{clemency}
  \E | \Delta M _n |^p &= \E \Big| \Delta X _n - \int\limits _{\substack{s \in (n,n+1], m \in \N, \\ u \in \R _+}}
  		m \1 \Big\{ u \leq 1\wedge \sum\limits _{k \in \Z _+} \frac{ \xi _{s-} (k)}{(k+m)^{2\alpha} } \Big\} ds \#(dm) du \Big| ^p \nonumber \\
  	&\leq C\, \E \int\limits _{\substack{s \in (n,n+1], m \in \N, \\ u \in \R _+}} m^p \1 \Big\{ u \leq 1\wedge \sum\limits _{k \in \Z _+} \frac{ \xi _{s-} (k)}{(k+m)^{2\alpha} } \Big\} ds \#(dm) du \nonumber \\
 		&= C \int\limits _{n} ^{n+1}ds \sum\limits _{m \in \N } m^p \Big( 1\wedge \sum\limits _{k \in \Z _+} \frac{\E \xi _{s-} (k)}{(k+m)^{2\alpha} }\Big) \leq  C \int\limits _{n} ^{n+1}ds \sum\limits _{m \in \N } m^p \sum\limits _{k \in \Z _+ } \frac{k + 1}{(k+1)^{\alpha} m^{\alpha} } \nonumber \\
 &= C \sum\limits _{m \in \N } \frac{1}{m^{\alpha - p}} \times \sum\limits _{k \in \N } \frac{1}{k^{\alpha - 1}}.
 \end{align}
Hence
\begin{equation*}
 \E \left| \Delta M _n \right| ^p < C < \infty,
\end{equation*}
where $C$ does not depend on $n$.
\end{proof}

We are now ready to prove the main result of this section. 
We will need to following form of the strong law of large numbers for martingales,
which is an abridged version of \cite[Theorem 2.18]{HH80}.

\begin{thm}\label{thm 2.18 HH80}
   Let $\{S_n = \sum\limits _{i = 1} ^n x _i , n \in \N  \}$ be an $\{\mathscr{F}_n\}$-martingale
   and $\{U_n\}_{n \in \N }$ be a non-decreasing sequence of positive real numbers, $\lim\limits_{n \to \infty} U_n = \infty$.
   Then for $p \in [1,2]$ we have
   \[
   \lim\limits_{n \to \infty } U _n ^{-1} S _n = 0
   \]
   a.s. on the set $\left\{ \sum\limits _{i = 1} ^\infty U _n ^{-p} \Ec{|x_i|^p}{\mathscr{F}_{i-1}} < \infty  \right\}$.
\end{thm}
 \begin{thm}[Linear speed]\label{at most linear}
	There exists $\bar C >0$
	such that a.s.
	\begin{equation}
	X_t \leq \bar C t
	\end{equation}
	for sufficiently large $t$.
\end{thm}
 \begin{proof}
Note that a.s.
\[
	\sum \frac{\mathbb{E}\left( |\Delta M_n|^p \vert \mathcal{F}_{n-1} \right)}{n ^p} < \infty,
\]
since by Proposition \ref{fluty} 
\begin{equation}\label{for Pasha}
	\E \sum \frac{\mathbb{E}\left( |\Delta M_n|^p \vert \mathcal{F}_{n-1} \right)}{n ^p} = \sum \frac{\mathbb{E} |\Delta M_n|^p  }{n ^p} < \infty.  
\end{equation}
Then Proposition \ref{fluty} and
 Theorem \ref{thm 2.18 HH80}, where we take $S_n = M_n$, 
 $U_n = n$, and $p =\frac{\alpha}{ 2} \wedge 2$, imply that a.s.
\begin{equation}
	\frac{M_n}{n} \to 0,  \quad n \in \N.
\end{equation}
Hence Proposition \ref{dibber}, $(ii)$,  yields that a.s. for large $n$ 
\begin{equation}\label{walk out on}
 \frac{X_n}{n}  = \frac{M_n}{n} + 
 \frac{Q_n}{n} \leq C_{_X} ,
\end{equation}
 where $C_{_X} >0$ is independent of $n$.
 
 Since $X_t$ is non-decreasing, \eqref{walk out on}
 holds for continuous parameter too if we increase 
 the constant: a.s.
 for large $t$, 
 \begin{equation}\label{walk out on2}
 \frac{X_t}{t}  \leq C_{_X} +1  .
\end{equation}
\end{proof}

\section{Superlinear growth for $\alpha \in (\frac 12, 2]$ in the discrete-space settings} \label{section micro alpha less than 2}

Our aim in this section is to prove the discrete-space equivalent of Theorem  \ref{thm superlinear growth alpha < 2}.
This is done in Theorem \ref{thm superlinear growth alpha <2 discrete}.
In Section \ref{section micro cont} we use Theorem \ref{thm superlinear growth alpha <2 discrete}
to prove Theorem \ref{thm superlinear growth alpha < 2}.
 The idea of the proof of Theorem \ref{thm superlinear growth alpha <2 discrete}
  is to find a certain system growing slower than our system, and then estimate the probability of
births outside an interval linearly growing with time.

Let $(\eta _t)$ be the  birth process on $\Z _+ ^{\Z }$
with birth rate \eqref{motile}, \eqref{toad},
but with $\alpha \in (\frac 12, 2]$.
As in Section \ref{section micro discr},
$(\eta _t)$ can be obtained as the unique solution 
to \eqref{se}. 
We focus here on the positive half line because it is sufficient for our purposes.

The next lemma has an auxiliary character and is a straightforward application of the large deviations principle.

\begin{lem} \label{lem large dev Poiss}
	Let $\chi$ be a Poisson random variable with mean $\lambda >0$.
	Then for large $\lambda$,
	\begin{equation}
	\P (\chi \leq \frac \lambda 3) \leq e^{-\frac \lambda 6 }.
	\end{equation}
\end{lem}
\begin{proof}
	Assume first that $\lambda \in \N$. Then 
	\[
	\chi \overset{d}{=} \chi _1 + ... + \chi _\lambda, 
	\]
	where $\chi_1, \chi _2, ...$ are i.i.d Poisson random variables with mean $1$.
	The cumulant-generating function of $\chi _1$ is 
	$
	\Lambda (u) = e^u - 1,
	$
	and the corresponding rate function 
	\[
	\Lambda ^ * (x) = \sup\limits _{u \in \R } (ux - \Lambda(u)) = x \ln x - x + 1, \ \ \ x \geq 0.
	\]
	
	By the large deviations principle, see e.g. \cite[Theorem 27.5]{KallenbergFound},

	\[
	\limsup\limits _{\lambda \to \infty} \lambda ^{-1} \ln \P (\frac1 \lambda \sum\limits_{i = 1} ^ \lambda \chi _i \leq \frac 1 3) 
	\leq - \inf\limits _{x \in [0, \frac 13]} \Lambda ^ * (x) 
 = - \frac  {2 - \ln 3}{3}.  
	\]
	Hence for large $\lambda$
	\[
	\ln \P (\chi \leq \frac \lambda 3) \leq - \frac  {2 - \ln 3}{3} \lambda + o(\lambda) < - \frac  {\lambda}{4}, 
	\]
	which gives the desired result for $\lambda \in \N$.
	The statement for $\lambda \notin \N$ follows by considering a Poisson random variable 
	with mean $ \lfloor \lambda \rfloor$ and noting that for large $\lambda$,
	$\frac  {\lfloor \lambda \rfloor}{4} > \frac \lambda 6$.
	
\end{proof}

In the next lemma it is shown that $\eta _t$ dominates a `rectangle-like' 
configuration, at least for large $t$.

\begin{lem}\label{slovenly}
	A.s. for sufficiently large $t$
	\begin{equation}
	\eta _t (x) \geq  \frac {t}{10}. 
	\end{equation}
	for all $x \in \Z \cap [0, \frac t4]$.
\end{lem}
\begin{proof}
	Let $(\gamma _t)$ be another birth process 
	with birth rate 
	\begin{equation}\label{birth rate gamma}
	b^{(\gamma)}(x,\eta) = 1 \wedge \left( \eta(x) + \eta(x-1) -1 \right)_+, \ \ \ x \in \Z ^1, \eta \in \Z _+ ^{\Z}.
	\end{equation}
	where $\kappa _+  = \max(\kappa, 0)$, and the initial condition $
	\gamma _0(k) = \1 \{ k = 0 \}$, $ k \in \Z  $. Alternatively, 
	\begin{equation}
	b^{(\gamma)}(x,\eta) = \begin{cases}
	1, \text{ if  }\eta(x) + \eta(x-1) > 0, 
	\\
	0, \text{ otherwise .}
	\end{cases}
	\end{equation}
	The process $(\gamma _t)$ can be obtained as a unique solution to \eqref{se}
	with the birth rate $b^{(\gamma)}$ instead of  $b^{(d)}$.

	We have 
	\begin{equation}\label{faucet}
	b^{(\gamma)}(x,\eta) \leq b^{(d)}(x,\eta), \ \ \ x \in \Z, \eta \in \Z _+ ^{\Z}.
	\end{equation}
	Using \eqref{faucet}, it is not difficult to show that a.s. for all $t\geq 0$,
	\begin{equation}\label{sulky}
	\gamma _t \leq \eta _t.
	\end{equation}
	In the continuous-space settings, the fact that \eqref{faucet}
	implies \eqref{sulky} is proven in \cite[Lemma 5.1]{shapenodeath}.
	In our case here we can take exactly the same proof.

	Let $\tau (n) := \inf \{t: \gamma _t(n) > 0 \}$
	be the time when $n \in \Z _+$ becomes occupied for $(\gamma_t)$.
	 Note that a.s. 
	$\tau (1) < \tau(2)< \dots $.
	Let $X^{(\gamma)}_t: = \max\{n: \gamma_t (n) > 0 \}$
	be the position of the rightmost occupied site for $(\gamma_t)$.
	The process
	$(X^{(\gamma)}_t)$ is a counting Markov process (that is, having  unit jumps only)
	and with jump rate constantly being $1$. 
	Therefore $( X^{(\gamma)}_t)$ is a  Poisson process
	whose $n$-th jump time coincides with  
	$\tau (n)$.
	By the law of large numbers, a.s.  for large $n$,
	\begin{equation*}
	\tau (n) < \frac 54 n.
	\end{equation*}

	Therefore, a.s. for large $n$  for $x \in \{0,1,...,n\}$, $t \in [\frac 54 n, 2n]$, we have  $b^{(\gamma)}(x,\eta_t) = 1$.
	By \eqref{se} (recall that the birth rate for $( \gamma_t)$ is $b^{(\gamma)}$  instead of $b^{(d)}$  )
	$$
	\gamma _{2n} (x) - \gamma_{\frac 54 n} (x) = \mathbf{N}\left(\left[\frac 54 n, 2n \right] \times \{x\} \times [0,1]\right),
	\ \ \
	x = 1,...,n,
	$$
	hence a.s. for large $n$
	\begin{equation} \label{maroon}
	\gamma _{2n} (x)  \geq \mathbf{N}([\frac 54 n, 2n] \times \{x\} \times [0,1]), \ \ \ x = 1,...,n.
	\end{equation}
	The random variables $ \omega ^{(n)}_x : =  \mathbf{N}([\frac 54 n, 2n] \times \{x\} \times [0,1])$, $x = 1,...,n$,
	are i.i.d Poisson with mean $\frac 34 n$. By Lemma \ref{lem large dev Poiss}
	for  $x = 0,1,...,n$,
	\begin{equation*}
	\PP{ \omega ^{(n)}_x < \frac n4 } \leq e ^{-\frac n8 },
	\end{equation*}
	hence
	\begin{equation} \label{omega ^n event}
	\PP{ \omega ^{(n)}_x < \frac n4  \text{ for some }  x \in \{ 1,...,n\} } \leq n  e ^{-\frac n8 }.
	\end{equation}
	Since $\sum\limits _{n \in \N}  n  e ^{-\frac n8 } < \infty$,
	by the Borel--Cantelli lemma the event in \eqref{omega ^n event}
	happens a.s. finitely many times only, therefore a.s. for sufficiently 
	large $n$
	\begin{equation*}
	\omega ^{(n)}_x \geq \frac n4 , \ \ \ \    x \in \{ 1,...,n\}.
	\end{equation*}
	By   \eqref{sulky}, \eqref{maroon}, and the definition of  $ \omega ^{(n)}_x$,
	 a.s. for sufficiently large $n$, 
	\begin{equation}\label{lapidation}
	\eta _{2n} (x) \geq \gamma _{2n} (x) \geq \frac n4,  \ \ \ x \in \{ 1,...,n\}.
	\end{equation}
	By
	 taking $n = \left \lfloor \frac t2 \right \rfloor - 1$ in \eqref{lapidation}
	we get a.s. for large $t $
	\begin{equation}
	\eta _{t} (x)  \geq \frac {t}{10},  \ \ \ x \in \{ 1,..., \left \lfloor \frac t2 \right \rfloor - 1 \},
	\end{equation}
	and the statement of the lemma follows.
\end{proof}

Now we are ready to prove the main result of the section.
\begin{thm}\label{thm superlinear growth alpha <2 discrete}
	For every $K_0, K_1 >0$ the set 
	\begin{equation*}
	\Big\{ t:  \sum\limits_{ \substack{ x \in \Z,  \\ x > K_0 + K_1 t}}\eta _t (x) > 0  \Big\}
	\end{equation*}
	is a.s. unbounded.
\end{thm}

\begin{proof}
	Without loss of generality we assume that $K_1 >1$.	
		Let $ \zeta _t  $ be defined by 
	$\zeta _t (k) = \left \lfloor \frac{t}{10} \right \rfloor \1 \{ 0\leq k \leq \frac t4 \}  $,
	hence  by Lemma \ref{slovenly}
	a.s. for large $t$,
	\begin{equation} \label{zeta leq eta}
	\zeta _t \leq \eta _t.
	\end{equation}
	
	Provided that $t$ is sufficiently large,
	the rate of a birth occuring inside  $[ K_0 + K_1 t + K_1, \infty )$
	at time $t$ is 
	\begin{equation}\label{quail1}
	\begin{gathered}
	\sum\limits _{\substack{ x \in \Z: \\
			x > K_0 + K_1 t + K_1}}  b^{(d)}(x, \eta _t) \geq  
	\sum\limits _{\substack{ x \in \Z: \\
			x > K_0 + K_1 t + K_1}} b^{(d)}(x, \zeta _t) 
	\geq 1 \wedge     \sum\limits _{\substack{ x \in \Z: \\
			x > K_0 + K_1 t + K_1}}
	\sum\limits _{k = 0   } ^{\left \lfloor \frac t4 \right \rfloor} 
	\zeta(k) a^{(d)}(x-k) .
		\end{gathered}
	\end{equation} 

Note that
	\begin{equation}\label{quail2}
\begin{gathered}
\sum\limits _{k = 0   } ^{\left \lfloor \frac t4 \right \rfloor} 
a^{(d)}(x-k)
=
\sum\limits _{k = 0   } ^{\left \lfloor \frac t4 \right \rfloor} 
\frac{1}{ |x-k|^{2\alpha}  } 
\geq  \sum\limits _{k = 0   } ^{\left \lfloor \frac t4 \right \rfloor} 
\frac{1}{ |x|^{2\alpha}  }
=
\left \lfloor \frac t4 \right \rfloor
\frac{1}{ |x|^{2\alpha}  }
\end{gathered}
\end{equation} 

hence for large $t$

	\begin{equation}\label{quail3}
\begin{gathered}
\sum\limits _{\substack{ x \in \Z: \\
		x > K_0 + K_1 t + K_1}}
\sum\limits _{k = 0   } ^{\left \lfloor \frac t4 \right \rfloor} 
 a^{(d)}(x-k)
\geq 
\left \lfloor \frac t4 \right \rfloor
\sum\limits _{\substack{ x \in \Z: \\
		x > K_0 + K_1 t + K_1}}
	\frac{1}{ |x|^{2\alpha}  }
	\\
	\geq 
   \frac{t}{5}  \times 
	\frac{1}{ 2(2\alpha - 1)|K_0 + K_1 t + K_1 + 1|^{2\alpha - 1}  }
	 \geq \frac{C}{t ^{2 \alpha -2}},
	\end{gathered}
	\end{equation} 
	 Since $\zeta (k ) = \left \lfloor \frac{t}{10} \right \rfloor$, $k \in 0,1,...,\left \lfloor \frac t4 \right \rfloor$,
	 by \eqref{quail1} and \eqref{quail3},
	 
	 		\begin{equation}\label{quail}
	 \begin{gathered}
	 	\sum\limits _{\substack{ x \in \Z: \\
	 		x > K_0 + K_1 t + K_1}}  b^{(d)}(x, \eta _t) \geq  
	 1 \wedge \left(  \left \lfloor \frac{t}{10} \right \rfloor   \frac{C}{t ^{2 \alpha -2}}\right) 
	 \geq 1 \wedge \frac{c}{t ^{2 \alpha -3}},
	 \end{gathered}
	 \end{equation} 
	 	where $c >0$ is a constant depending on $K_0, K_1, \alpha $,
	 but not on time $t$.

	Let $ \mathrm{L}_t $ be the number of jumps for $(\eta _t)$ that have occured prior $t$
	to the right of a growing interval $ [0, K_0 + K_1 s]$ for some $s \leq t$, that is, 
	\begin{align}
	\mathrm{L}_t  = &
	\#  \{(s,k): s \in  (0,t] , k \in ( K_0 + K_1 s, \infty )\cap \Z, \eta _s(k) - \eta _{s-}(k) = 1 \}
	\notag
	\\
	\label{mathrm K_n def}
	= &
	\int\limits _{ \substack{ \{(s,k,u): s \in  (0,t] , \\ k \in ( K_0 + K_1 s, \infty )\cap \Z, u \in [0,1] \} }   }
	\1_{ [0,b^{(d)}(k, \eta _{s-} )] } (u)
	\mathbf{N}(ds dk du).
	\end{align}

	Let $n \in \N$. We have 
	\begin{align}
	\mathrm{L}_{n+1} - \mathrm{L}_{n}  
	= &
	\int\limits _{ \substack{ \{(s,k,u): s \in  (n,n+1] , \\ k \in ( K_0 + K_1 s, \infty )\cap \Z, u \in [0,1] \} }   }
	\1_{ [0,b^{(d)}(k, \eta _{s-} )] } (u)
	\mathbf{N}(ds dk du). \notag
	\\
	\label{icky}
	\geq & \int\limits _{ \substack{ \{(s,k,u): s \in  (n,n+1] , \\ k \in ( K_0 + K_1 n + K_1, \infty )\cap \Z, u \in [0,1] \} }   }
	\1_{ [0,b^{(d)}(k, \eta _{n} )] } (u)
	\mathbf{N}(ds dk du) .
	\end{align}
	
	Define the sequence of independent random variables $\{F_n\}_{n \in \N}$,
	\begin{equation}\label{icky2}
	\begin{gathered}
	F_{n} :=
	 \int\limits _{ \substack{ \{(s,k,u): s \in  (n,n+1] , \\ k \in ( K_0 + K_1 n + K_1, \infty )\cap \Z, u \in [0,1] \} }   }
	\1_{ [0,b^{(d)}(k, \zeta _{n} )] } (u)
	\mathbf{N}(ds dk du) .
	\end{gathered}
	\end{equation}
	By \eqref{zeta leq eta} and \eqref{icky}, a.s. for large $n$
	\begin{equation}\label{icky3}
	\mathrm{L}_{n+1} - \mathrm{L}_{n}  \geq F_n.
	\end{equation}
	Since $\zeta _n$ is a non-random element of $\Z _+ ^ \Z$,
	$F_n$ is a Poisson random variable with mean
	$$m_n:=\sum\limits _{\substack{ x \in \Z: \\
			x > K_0 + K_1 n + K_1}}  b^{(d)}(x, \zeta _n).$$
	As we saw in \eqref{quail}, $m_n \geq 1 \wedge c n^{-(2 \alpha - 3)} $
	large $ n$.
	Hence, at least for large $n$,
	\begin{equation}\label{high jinks}
	\begin{gathered}
	\PP{F_n \geq 1} 
	=
	{ 1 - e^{-m_n} } \geq    1 - \exp\left\{ - 1 \wedge c n^{-(2 \alpha - 3)}  \right\}.
	\end{gathered}
	\end{equation}
	Recall that  for $c_1, c_2 \in \R$, $- c_1 \wedge c_2 = - (c_1 \wedge c_2)$.
	The series 
	\begin{equation} \label{series which diverges for alpha < 2}
		\sum\limits _{n\in \N}  \left( 1 - \exp\left\{ - 1 \wedge c n^{-(2 \alpha - 3)}   \right\}\right)
	\end{equation}
	 diverges  since $2 \alpha -3 \leq 1 $.
	Hence by  the Borel--Cantelli lemma 
	and \eqref{high jinks}, 
	\begin{equation}\label{enclave}
	\begin{gathered}
	\PP{F_n \geq 1 \text{ for infinitely many } n \in \N } 
	=
	1.
	\end{gathered}
	\end{equation}

	Finally, by \eqref{icky3} and \eqref{enclave}, 
	\begin{equation}
	\PP{\mathrm{L}_{n+1} - \mathrm{L}_{n} \geq 1 \text{ for infinitely many } n\in \N } = 1. 
	\end{equation}
	Recalling the definition of $\mathrm{L}_{n}$ in \eqref{mathrm K_n def},
	we see that our theorem is proven.
	
\end{proof}

\section{Continuous-space model}\label{section micro cont}
 
 We now return to the continuous-space model
 with the birth rate \eqref{motile c}
 described in the introduction. To prove Theorem \ref{finite speed of propagation} 
 and Theorem \ref{thm superlinear growth alpha < 2},
 we couple the continuous-space process with the discrete-space process from Sections \ref{section micro discr}
 and \ref{section micro alpha less than 2}
 and make use of Theorem \ref{at most linear} and Theorem \ref{thm superlinear growth alpha <2 discrete}.

 The continuous-space birth process defined by \eqref{motile c} and \eqref{stampede} can be obtained as a unique solution 
 to the stochastic equation 
 \begin{equation} \label{se continuous}
\begin{aligned}
|\eta _t \cap B| = \int\limits _{(0,t] \times B \times [0, \infty )   }
& \1 _{ [0,b^{(c)}(x, \eta _{s-} )] } (u) 
N^{(c)}(ds,dx,du)
+ |\eta _0 \cap B|,  \quad
t \geq 0, \  B\in \mathscr{B}(\R ^1),
\end{aligned}
\end{equation}
 where
$(\eta _t)_{t \geq 0}$ is a cadlag $\Go$-valued solution
 process,
 $N^{(c)}$ is a  Poisson point process on 
$\R_+  \times \R^1 \times \R_+  $,
the mean measure of $N^{(c)}$ is $ds \times dx \times du$,
and $\eta _0  = \{ 0\}$.
Equation \eqref{se continuous} is understood in the sense that the equality
holds a.s. for every bounded
$B \in \mathscr{B} (\R ^1) $ and $t \geq 0$.
In the integral on the right-hand side of \eqref{se continuous},
$x$ is the location and 
$s$ 
is the time of birth of a new particle. Thus, 
the integral over $B$ from $0$ to $t$ represents the number 
of births inside $B$ which occurred before $t$
(see \cite{shapenodeath} for more details).
The birth rate $b^{(c)}$ is as in \eqref{motile c}
with $a$ defined in \eqref{stampede}.

In this section we denote the solution 
to \eqref{se continuous} by $(\eta _t ^{(c)})$ with
the upper index `$(c)$' standing 
for `continuous'. We compare 
$(\eta _t ^{(c)})$ to the solution $(\eta _t ^{(d)})$
($(d)$ for `discrete')
of another equation 
 \begin{equation} \label{se discrete}
 \eta _t (k) = 
 \int\limits _{(0,t] \times \{k\} \times [0, 1 ]}
 \1_{ [0, C_\alpha  b ^{(d)}(i, \eta _{s-} )] } (u)
  {N}^{(d)}(ds di du)
 + \eta ^{(d)} _0 (k),
\end{equation}
which is of the form \eqref{se}
but with the birth rate multiplied by $C_\alpha > 0$:
\begin{equation}\label{motile d}
C_\alpha b^{(d)}(x,\eta) = C_\alpha \wedge \left( C_\alpha \sum\limits _{y \in \Z  } 
\eta(y) a ^{(d)}  (x-y) \right), \quad x \in \Z, \quad \eta \in \Z _+ ^{\Z},
\end{equation}
 with $a^{(d)}$ as in \eqref{toad}
and $\eta ^{(d)} _0 (k) = \1 \{ k = 0 \}$, 
and with 
 the driving Poisson point  process
\begin{equation*}
{N}^{(d)}([0,t] \times \{ k \} \times [0,u])
=
N ^{(c)}([0,t] \times (k - \frac 12, k + \frac 12]
\times [0,u]).
\end{equation*}

Note that $(\eta _t ^{(d)})$ is the process from the previous section evolving $C_\alpha$ times
 faster in time (or slower if $C_\alpha < 1$), and Theorem \ref{at most linear} applies to $(\eta _t ^{(d)})$ too.

Define also the discretization of the continuous-space process 
$(\eta _t ^{(c)}) 
$ as the process $(\eta _t ^{(dc)})$ 
taking values in $Z _+ ^{\Z}$
and 
\begin{equation}
 \eta _t ^{(dc)} (k) = 
 |\eta _t ^{(c)} \cap (k - \frac 12, k + \frac 12]|,  \ \ \ k \in \Z.
\end{equation}
Recall that for $c_1, c_2, c_3 \in \R$, $c_1c_2 \vee c_3 = (c_1c_2)\vee c_3$,
and the same for $\wedge$.
\begin{prop}\label{spastic}
	\begin{itemize}
		\item[$(i)$] Let $C_\alpha \geq c_\alpha 2 ^\alpha \vee 2$.
Then a.s. for all $t \geq 0$
\begin{equation}\label{iffy}
  \eta _t ^{(dc)} (k) \leq \eta _t ^{(d)}(k),
   \ \  \ k \in \Z.
\end{equation}
	 \item[$(ii)$] Let $ C_\alpha \leq c_\alpha 4^{-\alpha} \wedge \frac 12$. Then a.s. for all $t \geq 0$
	 \begin{equation}\label{iffy2}
	 \eta _t ^{(dc)} (k) \geq \eta _t ^{(d)}(k),
	 \ \  \ k \in \Z.
	 \end{equation}
	
	\end{itemize}
\end{prop}
\begin{proof}
	We start with
	$(i)$.
The proof will be done by induction on the 
birth moments of $(\eta _t ^{(c)})$.
Let $\{\theta _k \}$ be the moment of $k$-th birth
for $(\eta _t ^{(c)})$, $\theta _0  = 0$.
For $t = \theta _0$, \eqref{iffy} is satisfied.
For $x \in \R$, 
let 
here $\text{round}(x)$ is the closest integer to $x$,
with convention  that $\text{round}
(m + \frac 12) = m$, $m  \in \Z$.
It is sufficient to show that
that if a birth occurs for $(\eta _t ^{(c)})$
at time $\theta$ at $x \in \R$,
then a birth also occurs  for $(\eta _t ^{(d)})$
at $\theta$ at $\text{round}(x)$.
Assume \eqref{iffy} holds for $k < n \in \N$
and let $x_n$ be the place of birth 
at time $\theta _n$. Since $(\eta _t ^{(c)})$
solves
\eqref{se continuous}, we have a.s.
\[
 N ^{(c)} (\{\theta _k \} \times \{x_n \} \times 
  [0, b^{(c)} (x_n, \eta ^{(c)}_{\theta _k -}) )) = 1.
\]
Since $\frac{1 \vee | \text{round}(x)|^{2\alpha}}{(1 + |x|^2)^ \alpha} \leq 2 ^ \alpha$
	for $x \in \R$, we have 
	\begin{equation}\label{a leq a^d}
	a(x) \leq c_\alpha 2 ^\alpha a^{(d)}(\text{round}(x)), \ \ \ x \in \R,
	\end{equation}
and hence by the induction assumption a.s.
\[
b^{(c)} (x_n, \eta ^{(c)}_{\theta _k -}) 
\leq C_\alpha b^{(d)} 
(\text{round}(x_n), \eta ^{(dc)}_{\theta _k -}).
\]
Consequently, we also have a.s.
\[
 {N}^{(d)}(\{\theta _k \} \times \{
 \text{round}(x_n) \} \times 
  [0, C_\alpha b^{(d)} (\text{round}(x_n), 
  \eta ^{(d)}_{\theta _k }) )) = 1,
\]
and so we also have a birth 
for $( \eta ^{(d)} _t)$
at time $\theta _k$ at $\text{round}(x_n)$
since $  \eta ^{(dc)}_{\theta _k -}
 \leq  \eta ^{(d)}_{\theta _k- } $ and thus 
\eqref{iffy} holds at $\theta _n$ as well.

	 The proof of  $(ii)$ can be done by induction
	 on the birth moments of  $(\eta _t ^{(d)})$,
	 following exactly the same steps as the proof of $(i)$,
	 so we omit it.	 
	 We just point out that the counterpart of \eqref{a leq a^d}: 
	 \begin{equation*}
	 a(x) \geq c_\alpha 4^{-\alpha}  a^{(d)}(\text{round}(x)), \ \ \ x \in \R.
	 \end{equation*}

\end{proof}

\begin{proof}[Proof of Theorem \ref{finite speed of propagation}]
	The statement of the theorem follows 
	from Theorem \ref{at most linear} and Proposition \ref{spastic}, $(i)$.
\end{proof}
\begin{proof}[Proof of Theorem \ref{thm superlinear growth alpha < 2}]
	The statement of the theorem is a consequence of Theorem \ref{thm superlinear growth alpha <2 discrete}
	and Proposition \ref{spastic}, $(ii)$.
\end{proof}



\newcommand{\X}{{\mathbb{R}^1}}
\renewcommand{\S}{{S^{d-1}}}

\newcommand{\x}{\mathcal{X}}
\newcommand{\tx}{{\tilde{\mathcal{X}}}}
\newcommand{\D}{\mathscr{D}}
\newcommand{\Mn}{{\mathcal{M}}}
\newcommand{\K}{\mathscr{K}}
\renewcommand{\L}{\mathfrak{L}}
\newcommand{\Tauout}{\mathscr{O}}
\newcommand{\Tauin}{\mathscr{C}}
\renewcommand{\m}{{\mathfrak{m}}}
\newcommand{\A}{{\mathfrak{a}}}
\newcommand{\h}{{\mathfrak{h}}}
\newcommand{\T}{\mathfrak{t}}
\newcommand{\An}{\mathcal{A}}

\newcommand{\df}{\coloneqq}
\renewcommand{\Re}{\mathrm{Re}\,}
\renewcommand{\Im}{\mathrm{Im}\,}
\newcommand{\dist}{\mathrm{dist}\,}
\newcommand{\Lip}{\mathrm{Lip}}
\newcommand{\supp}{\mathrm{supp}\,}
\newcommand{\inter}{{\mathrm{int}}}
\newcommand{\ra}{\rightarrow}
\renewcommand{\varpi}{\omega}

\renewcommand{\EE}{L^\infty(\R)}

\let\copyint\int

\RenewDocumentCommand \int {o o}
{ \IfNoValueTF {#2} { \IfNoValueTF {#1} { \copyint } { \copyint\limits_{#1} } }	{ \copyint\limits_{#1}^{#2} } }

\NewDocumentCommand \diff {m m}
{ \frac{\partial #1}{\partial #2} }

\section{Mesoscopic equation}\label{section meso}
In this section we study the long time behavior of non-negative bounded solutions to the following nonlinear nonlocal evolution equation 
\begin{equation}\label{eq:basic_meso}
\begin{cases}
\begin{aligned}
\dfrac{\partial u}{\partial t}(x,t) =& \ \min\{ (a*u)(x,t), 1 \}, &\qquad  &x\in\R,\ t\in (0,\infty),\\
u(x,0) =& u_0(x), &\qquad 	&x\in\R.
\end{aligned}
\end{cases}
\end{equation}
Here $u \in C(\R_+,L^\infty(\R))\cap C^1((0,\infty), L^\infty(\R))$ is a classical solution to \eqref{eq:basic_meso}, $u_0\in L^\infty(\R,\R_+)$ is an initial condition;
the function $a\in L^1(\R):=L^1(\R,dx)$ is a probability density, i.e. $a(x)\geq0$ for a.a.\! (almost all) $x\in\R$ and
\begin{equation}\label{eq:normed}
\int_{\mathbb{R}} a(x)\,dx=1;
\end{equation}
the symbol $*$ stands for the convolution in $x$ on ${\mathbb{R}}$, i.e.
\[
(a*u)(x,t):=\int_{\mathbb{R}} a(x-y)u(y,t)\,dx.
\]

\textbf{\emph{An informal scaling and link between the microscopic
	and mesoscopic models.}}
\begin{em}
Here we describe the heuristic arguments which connect the birth process defined by \eqref{motile c} and \eqref{stampede} and the solution to the equation \eqref{eq:basic_meso}. 
We follow here the line of thought from \cite[Theorem 5.3]{fournier2004microscopic}.
Let us stress that we do not in any way give a rigorous proof of the link.

For a bounded measurable function $\phi:\Gamma_0\to\R$ consider the birth rate
\begin{equation}\label{pernicious}
	 b^n(x,\eta) = n\wedge\Big( \sum\limits_{y\in\eta} a(x-y)\Big),
\end{equation}
and the corresponding spatial birth process $(\eta^n_t)_{t\geq0}$.

 For $t \geq 0$,
 let  $\nu^n_t$
 be a random purely atomic measure on  $\R $ defined by
 \[
  \nu _t ^n (A) = |\eta ^n _t \cap A|.
 \]

The intuition is that considering $(\eta^n_t)_{t\geq0}$ and $(\nu^n_t)_{t\geq0}$ we increase the birth rate 
but then we are going to  rescale the process by multiplying by $\frac 1n $ to compensate for the increase in the number of particles.
Let $\mathscr{M}(\R )$ be the space of finite non-negative measures
equipped with the vague topology.
Assume that  if $\frac 1n \nu_0^n(dx)$ converges in law to a deteministic measure
 $\mu_0(dx)$, 
 then the measure valued function $\frac 1n \nu_t^n(dx)$ converges in law in the Skorokhod space
 $\mathrm{D}([0,T], \mathscr{M}(\R ))$
   to a deterministic $\mathscr{M}(\R )$-valued function 
 $t \mapsto \mu _t $. Since  
  \eqref{terse} below
  is  a martingale with a vanishing quadratic variation,
  this limiting measure-valued function
 should then be  
  a unique solution to the integral equation
 written in the weak form:
 \begin{equation}\label{egregious}
 \langle  \mu _t, f \rangle =  \langle  \mu _0, f \rangle 
+ 
\int\limits _0 ^t ds \int\limits_{x \in \R} f(x) 
\min \{1,  \int\limits_{y \in \R} a(x - y) \mu _t(dy)  \} dx.
 \end{equation}

Assume furthermore that $\mu _t$ has a density with respect to the Lebesgue measure
provided that the initial condition does: $\mu_0(x) = u_0(x) dx$.
We denote the density of $\mu _t$ by $u(t,x)$, so that $\mu _t(dx) = u(t,x)dx$.
Denote $u_t = u(t,\cdot)$, ($u_t$ is a function on $\R$).
Then we have 
\[
\frac 1n \sum\limits _{y \in \eta ^n _t } 
a (x-y) \to (a \ast u_t)(x)
\]
and hence, assuming that $u$ is differentiable,
\[
\frac{\partial u(t,x)}{\partial t} (t,x)
= \lim\limits _n \frac 1n b ^{(n)} (x,\eta ^n _t)
= \lim\limits _n \frac 1n \left[ 
n \wedge \left( \sum\limits _{y \in \eta ^n _t  } 
a (x-y) \right)
\right] 
=
\lim\limits _n  
1 \wedge \left( \frac 1n \sum\limits _{y \in \eta ^n _t  } 
a (x-y) \right) 
\]
\[
= 1 \wedge \big( (a \ast u_t)(x) \big),
\]
which coincides with \eqref{eq:basic_meso}.

The proof that the limiting measure
is indeed the unique solution to 
\eqref{egregious}
 would have to rely on the martingale properties of the spatial birth processes.
The generator of the birth process with the rate \eqref{pernicious} 
is 
\begin{equation*}
(L^n\phi)(\eta) = \int_\R b^n(x,\eta) [\phi(\eta\cup x)-\phi(\eta)] dx.
\end{equation*}
As in  \cite{fournier2004microscopic},
  one could show that for any bounded measurable $f:\R\to\R$ 
\begin{equation} \label{terse}
M_t^{n,f} := \frac 1n \int_\R f(x) \nu_t^n(dx) - \frac 1n \int_\R f(x) \nu_0(dx) - \frac 1n \int_0^t \int_\R \Big[n\wedge \int[\R] a(x-y)\nu_t^n(dy)\Big]f(x)dxds
\end{equation}
is a c\`{a}dl\`{a}g martingale with the quadratic variation
\[
\langle M^{n,f} \rangle_t = \frac {1}{n^2} \int_0^t \int_\R \Big[n\wedge \int[\R] a(x-y)\nu_t^n(dy)\Big] f^2(x)dxds.
\]
Hence
\[
\E |M_t^{n,f}|^2 = \E \langle M^{n,f} \rangle_t \leq \frac{ c_\alpha \|f \| \E |\eta _t ^n|}{n} 
\]
where $\|f\| = \sup\limits _{x \in \R} f(x)$. Thus $\E |M_t^{n,f}|^2 \to 0$ a.s. uniformly on any finite interval 
$[0,T]$, $n \to \infty$.  
\\
\end{em}

	The proof of Theorem \ref{thm:est_front} falls naturally into two parts.
First, we obtain an estimate of the solution $u$ from above (see Proposition \ref{prop:ba:bound_above_gen}), which implies that $u$ propagates at most exponentially.
Second, we construct subsolutions \eqref{eq:def_g_and_h} to \eqref{eq:basic_meso} in order to estimate `small' level-sets of the solution from below.
Then, locally uniform convergence of $u$ to infinity (Lemma \ref{lem:hair_trigger_effect}) demonstrates that the solution does not propagate slower than exponentially.

We start with general properties of the solutions to \eqref{eq:basic_meso}.
\begin{defn}
We call an operator $G$ in $L^\infty(\R)$ monotone, if for all $h_1, h_2 \in \EE$,
\begin{align*}
	h_1(x)\leq h_2(x),\ x\in\RR, \qquad &\Rightarrow \qquad Gh_1(x) \leq Gh_2(x),\ x\in\RR.
\end{align*}
We call an operator $G$ in $L^\infty(\R)$ Lipschitz continuous, if there exists $K>0$, such that for all $h_1, h_2 \in \EE$,
\begin{align*}
	\quad \|Gh_2 - Gh_1\|_{\EE} &\leq K \|h_2 - h_1\|_{\EE}.
\end{align*}
\end{defn}
\begin{rem}\label{rem:G_example}
	$Gu = \min\{a*u,1\}$ is a monotone and Lipschitz continous operator in $\EE$ with the Lipschitz constant $K=1$.
\end{rem}

 Since $G$ is Lipschitz-continuous in the Banach space $L^\infty(\R)$, well-posedness of \eqref{eq:basic_meso} 
is easily shown by a Picard iteration scheme (see e.g.~\cite[Chapter 6, Theorem 1.2, Theorem 1.7]{Pazy}).
For completeness we provide the details (cf.~\cite{some_pde}).
\begin{prop}\label{prop:exist_uniq}
	Let $G$ be Lipschitz continuous on $L^\infty(\R)$ and $u_0\in L^\infty(\R)$.
	Then for any $T>0$ there exists a unique  classical solution $u\in C(\R_+, L^\infty(\R))\cap C^1((0,\infty), L^\infty(\R))$  to the equation,
	\begin{equation}\label{eq:dudt_Gu}
		\left\{
			\begin{aligned}
				\frac{\partial u}{\partial t}(x,t) &= (Gu)(x,t), \qquad &t\in(0,\infty),\ x\in\X,\\
			u(x,0) &= u_0(x), \qquad &x\in\X.
			\end{aligned}
		\right.
	\end{equation}
\end{prop}
\begin{proof}
For $0\leq \tau < \Tau <\infty$, $v\in C([\tau,\Tau], L^\infty(\R))$, $w\in L^\infty(\R)$, we define,
	\begin{equation}\label{eq:Phi_def}
		(\Phi_w v)(x,t) := w(x) + \int[\tau][t] (Gv)(x,s)ds, \qquad t\in[\tau, \Tau],\ x\in\R.
	\end{equation}
  Let $\|v\|_{\tau,\Tau}:=\sup\limits_{t\in[\tau,\Tau]}\|v(\cdot,t)\|_\infty$.
  Then, one easily gets, that $\|\Phi_w v\|_{\tau,\Tau}<\infty$ and
  \[
  	\|\Phi_w v_1-\Phi_w v_2\|_{\tau,\Tau}\leq K (\Tau-\tau)\|v_1-v_2\|_{\tau,\Tau},
  \]
	where $K$ is the Lipschitz constant of $G$.
	Therefore, $\Phi_w$ is a contraction mapping on $C([\tau,\Tau], L^\infty(\R))$, provided that $\Tau-\tau<\frac{1}{K}$ 
	Fixing any  $\delta\in\bigl(0,\frac{1}{K}\bigr)$, one gets that there exists the limit $u$ of $(\Phi_w)^n v$, $n\to\infty$, for any $v$, on time intervals $[k\delta,(k+1)\delta]$, $k\in\N\cup\{0\}$, with the corresponding $w(x) = u(x,k\delta)$.
	Therefore, for any $0\leq\tau<\Tau$, we have that $u\in C([\tau,\Tau], L^\infty(\R))$ and
  \[
		u(x,t)=(\Phi_{u(\cdot,\tau)} u)(x,t), \quad t\in[\tau,\Tau].
  \]
	Since $G$ is Lipschitz continuous, then it follows that $u\in C(\R_+, L^\infty(\R))\cap C^1((0,\infty), L^\infty(\R))$ and it solves \eqref{eq:dudt_Gu}. The proof is completed.
\end{proof}

We introduce the following operators
\begin{align}\label{eq:def_Z_Q}
  Z_y v(x) &= v(x-y),&v\in\EE,\ y\in\RR,\\
	Q_t v(x) &= u(x,t),&t\geq0,\ x\in\RR, 
\end{align}
where $u(x,0)=v(x)$ and $u$ solves \eqref{eq:dudt_Gu}.
Thus $Z_y$ is a shift operator in $\R$, and $Q_t$ is the semiflow generated by \eqref{eq:dudt_Gu}. 
The following important property follows form the proof of Proposition \eqref{prop:exist_uniq}.
\begin{cor}\label{cor:Z_Q_are_commutative} 
	 If $Z_y$ and $G$ are commutative for all $y\in \R$, then the operators $Z_y$ and $Q_t$ are commutative, namely,
  \begin{equation}\label{eq:Z_Q_are_commutative}
  	Z_y Q_t = Q_t Z_y,\qquad y\in\RR,\ t\geq0.
  \end{equation}
\end{cor}
\begin{proof}
	Following the notations of the proof of Propostion \ref{prop:exist_uniq}, we have for $v\in C([0,\delta], L^\infty(\R))$, $u_0\in L_\infty(\R)$, $y\in\R$,	
	\[
		(Z_y\Phi_{u_0} v)(x,t) = (\Phi_{Z_y u_0} Z_yv)(x,t), \qquad x\in\R, t\in[0,\delta].
	\]
	Hence, we have, for $t\in[0,\delta]$, $y\in\R$,
	\[
		Z_y Q_t u_0 = Z_y \lim_{n\to\infty} \Phi^n_{u_0} v = \lim_{n\to\infty} \Phi^n_{Z_yu_0} Z_yv = Q_t Z_y u_0.
	\]
	Repeating the same argument on $[\delta, 2\delta], \cdots, [k\delta,(k+1)\delta], \cdots$, finishes the proof. 
\end{proof}
We denote, for $u\in C([0,T],L^\infty(\R))\cap C^1((0,T], L^\infty(\R))$,
\begin{equation}\label{eq:F_operator}
	\mathcal{F}u(x,t) := \dfrac{\partial u}{\partial t}(x,t) - Gu(x,t), \qquad x\in \R,\ t>0.
\end{equation}

\begin{prop}[Comparison principle]\label{prop:compar_princ}
	Let $G$ be monotone and Lipschitz on $\EE$, $T\in(0,\infty)$ be fixed and functions $u_{1},u_{2}\in C([0,T],L^\infty(\R))\cap C^1((0,T], L^\infty(\R))$, be such that, for any $(x,t)\in\X \times(0,T]$,
	\begin{gather}
		\mathcal{F}u_1(x,t) \leq \mathcal{F}u_2(x,t), \label{eq:max_pr_BUC:ineq}\\
	  0 \leq u_{1}(x,t), \qquad 0\leq u_{2}(x,t)\leq c, \qquad u_{1}(x,0)\leq u_{2}(x,0).\label{eq:init_for_compar}
	\end{gather}
	Then $u_{1}(x,t)\leq u_{2}(x,t)$, for all $(x,t)\in\X \times[0,T]$. In particular, $u_{1}\leq c$.
\end{prop}

\begin{proof}
  Define the following functions for $x\in\X, t\in(0,T], w\in\EE$,
  \begin{align}
		f(x,t) &:= \mathcal{F}u_2(x,t) - \mathcal{F}u_1(x,t) \geq 0, \label{dif_pos} \\
		F(x,t,w) &:= G(w+u_1)(x,t) - Gu_1(x,t) + f(x,t), \\
		v(x,t) &:= u_2(x,t) - u_1(x,t), \label{eq:v_def}
  \end{align}

  Clearly, $v\in C([0,T],L^\infty(\R))\cap C^1((0,T], L^\infty(\R))$, and it is straightforward to check that
  \begin{equation}\label{tautology}
    \frac{\partial}{\partial t} v(x,t) = F(x,t,v(x,t)),
  \end{equation}
  for all $x\in\X$, $t\in(0,T]$. Therefore, $v$ solves the following integral equation in $L^\infty(\R)$:
  \begin{equation}
  	\begin{cases}
  		\displaystyle v(x,t)=v(x,0)+\int_0^t F(x,s,v(x,s))ds, & \quad (x,t)\in\X {\times}(0,T],\\[3mm]
  		v(x,0)=u_{2}(x,0)-u_{1}(x,0), & \quad x\in\X,
  	\end{cases}\label{eq:max_pr_BUC:u2-u1_lin}
  \end{equation}
  where $v(x,0)\geq0$, by \eqref{eq:init_for_compar}.
  
  Consider also another integral equation in $L^\infty(\R)$:
  \begin{align}
	\tilde{v}(x,t)  &= (\Psi \tilde{v})(x,t), \qquad (x,t)\in\X\times(0,T], \label{eq:max_pr_BUC:pos_sol} \\
  \shortintertext{where}
	(\Psi w)(x,t) &:= v(x,0) + \int_0^t\max\{F(x,s,w(x,s)),0\}\,ds,\qquad w\in C([0,T],L^\infty(\R)). \label{eq:max_pr_BUC:Psi}
  \end{align}
	It is easily seen that $0\leq w\in C([0,T],L^\infty(\R))$ yields $0\leq \Psi w\in C([0,T],L^\infty(\R))$.
	Next, for any $\tilde{T}< T$ and for any $w_{1}$, $w_{2}$ from $C([0,\tilde{T}],L^\infty(\R,\R_+))$, one gets by \eqref{eq:max_pr_BUC:Psi} that
  \begin{align}
    \|\Psi w_{1}-\Psi w_{2}\|_{\tilde{T}} &\leq\tilde{T} K \|w_{2}-w_{1}\|_{\tilde{T}}, \label{contractioncomp}
  \end{align}
 where $K>0$ is the Lipschitz constant of $G$   and  we used  the elementary inequality $\lvert \max\{a,0\}-\max\{b,0\}\rvert\leq |a-b|$, $a,b\in\R$.
  Therefore, for $\tilde{T}<K^{-1}$, $\Psi$ is a contraction on $C([0,\tilde{T}],L^\infty(\R,\R_+))$. Thus, there exists a unique solution
  to \eqref{eq:max_pr_BUC:pos_sol} on $[0,\tilde{T}]$. In the same
  way, the solution can be extended on $[\tilde{T},2\tilde{T}], [2\tilde{T},3\tilde{T}]$, \ldots, and therefore, on the whole $[0,T]$.
  By \eqref{eq:max_pr_BUC:pos_sol}, \eqref{eq:max_pr_BUC:Psi},
  \begin{equation}\label{result}
  	\tilde{v}(x,t)\ge v(x,0)\ge0,
  \end{equation}
  hence, by \eqref{eq:max_pr_BUC:Psi},
  \begin{equation}\label{safas}
  \tilde{v}(x,t)=v(x,0)+\int_0^t F(s,\tilde{v}(x,s))\,ds=:\Xi(\tilde{v})(x,t).
  \end{equation}
	Since $0\leq\tilde{v}\in C([0,T],L^\infty(\R))$ and $G$ is monotone, \eqref{safas} implies that $\tilde{v}$ is a solution to \eqref{eq:max_pr_BUC:u2-u1_lin} as well.
	The same estimate as in \eqref{contractioncomp} shows that $\Xi$ is a contraction on $C([0,\tilde{T}],L^\infty(\R))$, for small enough $\tilde{T}$. Thus $\tilde{v}=v$ on $\X \times[0,\tilde{T}]$, and one continues  this consideration as before on the whole $[0,T]$.
	Then, by \eqref{result}, $v(x,t)\geq0$ on $\X \times[0,T]$,  and the statement of the proposition follows.
\end{proof}
Let us recall that $B_\sigma$ denotes the interval $[-\sigma,\sigma]$ and $\EEp$ is defined by \eqref{def:Lp_plus}.
\begin{lem}\label{lem:hair_trigger_effect}
	Let exists $\sigma>0$ such that $a(x) \geq \sigma,\ x\in B_\sigma$. Suppose also that $u_0\in \EEp$ and $u$ be the corresponding solution to \eqref{eq:basic_meso}. 

	Then for any $r>0$, the following limit holds 
	\begin{equation}\label{eq:hair_trigger}
		\lim_{t\ra\infty} \inf_{x\in B_r} u(x,t) \ra \infty. 
	\end{equation}
\end{lem}
\begin{proof}
	By assumptions of  the lemma,
	\[
		d(x) := \sigma \1_{B_\sigma} (x) \leq a(x),\quad x\in\RR.
	\]
 Since 
		 $u_0\in\EEp$,
		  there exist $\delta>0$, $x_0\in\R$, such that $u_0(x)\geq v_0(x) := \delta\1_{B_\delta(x_0)}(x)$, $x\in\R$. 
	Let $v$ satisfies 
	\[
		\diff{v}{t}(x,t) = (d*v)(x,t), \quad x\in\R,\ t>0; \qquad v(x,0) = v_0(x) \leq u_0(x).
	\]
	We define $Df:= d*f$. Since for any $r_1 \leq r_2$,
	\[
		\Big( \1_{B_{2r_1}}*\1_{B_{2r_2}} \Big)(x) \geq r_1 \1_{B_{2r_2+r_1}}(x),\quad x\in\RR,
	\]
	the following estimate holds
	\[
		\delta \sum_{j\geq 0} \big( \min\{\delta,\sigma\} \big)^j \frac{t^j \sigma^j}{2^jj!} \1_{B_{\delta+\sigma j / 2}}(x) \leq \sum_{j\geq 0} \frac{t^j D^j v_0 (x)}{j!} = v(x,t),\quad x\in\RR,\ t\geq0.
	\]
	Hence, for any $t>0$, $r>0$,
	\[
		\nu_t := \inf_{x\in B_{r+\sigma}} v(x,t) >0.
	\]
	Let us define,
	\[
		T:= \inf \{ t>0, \|v(\cdot,t)\|_\infty \geq 1 \} >0.
	\]
	By Proposition \ref{prop:compar_princ},  applied with $Gu = \min\{ a*u, 1\}$,
	\[
		u(x,t_0) \geq v(x,t_0) \geq \nu_{t_0}, \quad x\in B_{r+\sigma},\ t_0\in (0,T).
	\]
	Since $u\geq0$, then by \eqref{eq:basic_meso}, $u(x,t)$ is nondecreasing in $t$. Thus for all $t\geq t_0$, $x\in B_{r}$,
	\[
		\frac{\partial u}{\partial t}(x,t) = \min\{(a*u)(x,t),1\} \geq \min\{(a*u)(x,t_0),1\} \geq \min\{ \frac{\sigma \nu_{t_0}}{2} ,1\} >0.
	\]
	As a result, \eqref{eq:hair_trigger} holds. The proof is completed.
\end{proof}

From now on we study the case when $a(x)$ is defined by \eqref{stampede}, with $\alpha>\frac{1}{2}$.

\begin{lem}\label{lem:conv_of_long_tail_and_o_small}
	Let $a(x)$ be defined by \eqref{stampede} with $\alpha>\frac{1}{2}$ and $u_0\in\EEp$. Then there exists $R>0$ such that the following statements hold 
	\begin{enumerate}
		\item For all $|x|\geq R$,
			\begin{equation}\label{eq:a_conv_u0_est_comp_sup}
				|x|^{-2\alpha} \preceq a(x) \preceq (a*u_0)(x).
			\end{equation}
		\item If there exist $\mu>0,\ \rho\in\R,$ such that $u_0(x) \geq \mu,\ x\leq \rho,$ then for all $x\geq R$,
			\begin{equation}\label{eq:a_conv_u0_est_monotone}
				x^{-2\alpha+1} \preceq \inte{x}{\infty} a(y) dy \preceq (a*u_0)(x).
			\end{equation}
	\end{enumerate}
\end{lem}

\begin{proof}
	We start with the first part of the lemma.
 Without loss of generality we may assume that  $u_0 \in L^1(\RR)$.   

 By \eqref{def:Lp_plus}, there exist $\delta>0$ and $x_0\in\R$, such that $u_0(x)\geq \delta$, $x\in B_\delta(x_0)$. Since for any $r\geq |x_0|$, $a(x) \sim a(|x|+r)$ as $|x| \ra \infty$, then there exists $R>0$ such that the following estimate holds, for all $|x| \geq R$, 
  \[
		|x|^{-2\alpha} \preceq a(x) \preceq a(|x|+r) \inte{|y|\leq r}{} u_0(y)dy \leq \inte{|y| \leq r}{} a(x-y)u_0(y)dy \leq (a*u_0)(x), 
	\]
	Now we prove the second part of the lemma. By the assumptions on $u_0$, there exists decreasing smooth $v_0 \in \EEp$ such that $v_0(x) \ra 0$ as $x\ra\infty$, $v_0 \leq u_0$ and $\frac{ \partial v_0(x)}{\partial x} \leq 0$ is compactly supported.
		Then by the first part of the lemma applied to $-\frac{ \partial v_0(x)}{\partial x}$ instead of $u_0$, there exists $R>0$ such that
	\[
		x^{-2\alpha} \preceq a(x) \preceq -\big(a*\frac{\partial v_0}{\partial x}\big)(x),\quad x\geq R.
	\]
	Hence, for all $x\geq R$,				
  \[
		x^{-2\alpha+1} \preceq \inte{x}{\infty} a(y) dy \preceq -\inte{x}{\infty} \big(a*\frac{\partial v_0}{\partial y}\big)(y) dy = (a*v_0)(x) \leq (a*u_0)(x).
	\]
	The proof is completed.
\end{proof}

\begin{lem}\label{lem:subsol}
	Let $a$ be defined by \eqref{stampede} with $\alpha>\frac{1}{2}$, and we define 
	\begin{align}
		h(x,t) = \1_{\R_-}(x) + \min \bigl\{ 1, x^{-2\alpha+1} e^{(1-\eps)t} \1_{(0,\infty)}(x) \bigr\},\quad g(x,t) = \min\bigl\{1,|x|^{-2\alpha} e^{(1-\eps) t}\bigr\}. \label{eq:def_g_and_h}
	\end{align}
Then, for any $\eps\in(0,1)$ there exists $\tau_0=\tau_0(\eps)>0$ such that for all  $l>0$ the functions
	\begin{align} 
		H(x,t,l) := \frac{1}{l} \int[t][t+l] h(x,s)ds, \quad G(x,t,l):= \frac{1}{l} \int[t][t+l] g(x,s)ds, \label{eq:def_G_and_H}
	\end{align}
	are sub-solutions to $\partial_t u = a*u$ on $[\tau_0,\infty)$, 
Namely (c.f.  \eqref{eq:F_operator}), for all $l>0$, 
	\begin{equation*}
		\frac{\partial G}{\partial t} (x,t,l) \leq (a*G)(x,t,l),\qquad \frac{\partial H}{\partial t}(x,t,l) \leq (a*H)(x,t,l), \qquad x\in\RR,\ t\geq \tau_0. 
	\end{equation*}
	In this case one can understand $g$ and $h$ as `weak' sub-solutions to $\partial_t u = a*u$.
\end{lem}

\begin{proof}

	We denote $r_t = \exp(\frac{t-\eps t}{2\alpha-1})$.
	Note that $h(x,t) = 1 \Leftrightarrow x \leq r_t$.
	Since $t\to h(x,t)$ is absolutely continuous, then for all $x\in\R$ and almost all $t>0$, we have
	\begin{equation}\label{eq:h_subsol_sufficient}
		-\frac{\partial h}{\partial t}(x,t)  + (a*h)(x,t) = -(1-\eps) h(x,t) \1_{x\geq r_t} + (a*h)(x,t).
	\end{equation}
	Note that
	\begin{align}
		\frac{\partial H}{\partial t} (x,t,l) &= \frac{h(x,t+l)-h(x,t)}{l} = \frac{1}{l} \int[t][l+t] \frac{\partial h}{\partial t}(x,s)ds, \label{eq:dH_dt}\\
		(a*H)(x,t,l) &= \frac{1}{l} \int[t][l+t] (a*h)(x,s)ds. \label{eq:a_conv_H}
	\end{align}
	Hence, by \eqref{eq:dH_dt}, $H\in C(\R_+,L^\infty(\R))\cap C^1((0,\infty),L^\infty(\R))$.
Moreover, by \eqref{eq:dH_dt} and \eqref{eq:a_conv_H}, ${\partial_t h \leq a*h }$, for all $x\in\R$ and almost all $t>0$, yields ${\partial_t H \leq a*H}$, for $H$ as a vector valued function.
Thus, it is sufficient to check that the right-hand side of \eqref{eq:h_subsol_sufficient} is non-negative. 

 Take $\delta\in\bigl(0,1)$. There exists $x_0=x_0(\delta) > 0$, such that 
	\begin{equation}\label{eq:b_over_b_geq_delta1}
		\sup_{|y|\leq \sqrt{x}} \Big(\frac{x+y}{x}\Big)^{2\alpha-1}\geq 1-\delta,\qquad x\geq x_0.
	\end{equation}
	
	Let $\tau>0$ be such that $r_t\geq x_0$, $t\geq \tau$.
	By \eqref{eq:h_subsol_sufficient},  in order to show that $h$ is a subsolution, it is sufficient to prove 
	that there exists $t_0=t_0(\eps,\delta)>\tau$, such that
	\begin{equation}\label{eq:provethisnice}
		\frac{(a*h)(x,t)}{h(x,t)}\geq (1- \delta) \int_{-\sqrt{r_t}}^{r_t} a(y) dy
	\end{equation}
	for all $x\in \R$ and $t\geq t_0$. Note that, 
	\begin{equation}\label{eq:aplus_conv_g}
		(a*h)(x,t) \geq \int_{-\sqrt{r_t}}^{r_t} a(y) h(x-y,t)dy  
	\end{equation}
	for $x\in\R$ and $t>\tau$.

	1. Let $x\in(-\infty,r_t-\sqrt{r_t})$, $t>\tau$. 
	Since $h(x,t)=1$, for $x\leq r_t$, then we have  
	\begin{equation}
		\frac{(a*h)(x,t)}{h(x,t)} = \frac{\int_\R a(y) h(x-y,t)dy}{h(x,t)} \geq \int[\R] a(y) \1_{x-y\leq r_t}(y) dy = \int[x-r_t][\infty] a(y)dy \geq \int[-\sqrt{r_t}][\infty] a(y)dy,
	\end{equation}
	and \eqref{eq:provethisnice} holds. 

	2. Let $x\in[r_t-\sqrt{r_t},r_t)$, $t>\tau$. 
		Note that $h(x,t) = 1$, and $h(x-y,t) = 1$ for $y \geq x-r_t$. 
  Then \eqref{eq:aplus_conv_g} yields, that
	\begin{equation}
		\frac{(a*h)(x,t)}{h(x,t)} \geq  \int_{x-r_t}^{r_t} a(y) dy+ \int_{-\sqrt{r_t}}^{x-r_t} a(y) \Big(\frac{r_t}{x-y}\Big)^{2\alpha-1}dy. \label{eq:intermediate}
	\end{equation}
	Next, for the considered $x$, $-\sqrt{r_t}\leq y\leq x-r_t$ yields $0\leq x-y-r_t<\sqrt{r_t}$, and hence, by \eqref{eq:b_over_b_geq_delta1}, there exists $t_1>\tau$ such that for all $t\geq t_1$ and $x\in[r_t-\sqrt{r_t},r_t)$
	\[
		\Big(\frac{r_t}{x-y}\Big)^{2\alpha-1} = \Big(\frac{r_t}{r_t+(x-y-r_t)}\Big)^{2\alpha-1} \geq \Big(\frac{r_t}{r_t+\sqrt{r_t}}\Big)^{2\alpha-1} \geq 1- \delta,
	\]
that, together  with \eqref{eq:intermediate}, implies \eqref{eq:provethisnice}.

	3. Let $x\geq r_t$, $t>\tau$. Then, by \eqref{eq:aplus_conv_g},
	\begin{equation}
		\frac{(a*h)(x,t)}{h(x,t)} \geq \frac{x^{2\alpha-1}}{e^{(1-\eps) t}} \int_{x-r_t}^{r_t} a(y) dy 
		+\int_{-\sqrt{r_t}}^{x-r_t} a(y) \Big(\frac{x}{x-y}\Big)^{2\alpha-1} dy.\label{eq:intermediate2}
	\end{equation}
	Next, $e^{(1-\eps) t}=r_t^{2\alpha-1} \leq x^{2\alpha-1}$ for $t>\tau$.
	The latter also implies that $(x-y)^{2\alpha-1}\leq x^{2\alpha-1}$ if $0\leq y\leq x-r_t$.
Finally,  by \eqref{eq:b_over_b_geq_delta1}, there exists $t_2>t_1$, such that $x^{2\alpha-1} \geq (1-\delta)(x-y)^{2\alpha-1}$, if only $-\sqrt{r_t}\leq y<0$, $x\geq r_t$, $t\geq t_2$.
As a result, \eqref{eq:intermediate2} implies \eqref{eq:provethisnice}, which is proved hence for all $x\in\R$ and $t\geq t_2$.
The proof for $g(x,t)$ with $r_t = \exp(\frac{t-\eps t}{2\alpha})$ is similar.
\end{proof}

%
\begin{lem}\label{lem:a_conv_agamma_equiv_agamma}
	Let $a$ be defined by \eqref{stampede} with $\alpha>\frac{1}{2}$. Then for any $\gamma\in(\frac{1}{2\alpha},1)$ the following limit holds,
 \begin{equation}\label{eq:a_conv_a_gamma}
	 \frac{a*a^\gamma(x)}{a^\gamma(x)} \to 1, \qquad |x|\to \infty.
 \end{equation}
\end{lem}
\begin{proof} 
	Take arbitrary $\delta\in(0,1)$, $\gamma\in (\frac{1}{2\alpha},1)$. 
	Let us consider, for $x$ such that $|x|>2|x|^\delta$, a disjoint decomposition $\R=D_1(x)\sqcup D_2(x)\sqcup D_3(x)$, where

\begin{gather*}
	D_1(x):= [-|x|^{\delta}, |x|^{\delta}], \quad D_2(x):= (-\frac{|x|}{2}, -|x|^\delta) \cup (|x|^\delta, \frac{|x|}{2} ),\\
 D_3(x)=(-\infty, -\frac{|x|}{2}] \cup [\frac{|x|}{2}, \infty).
\end{gather*}
Then, $\frac{(a*a^\gamma)(x)}{a^\gamma(x)}=I_1(x)+I_2(x)+I_3(x)$, where
\[
	I_j(x):=\int_{D_j(x)}a(y) \Big( \frac{1+|x|^2}{1+|x-y|^2}\Big)^{\alpha\gamma}dy, \quad j=1,2,3.
\]
Using the inequality $|x-y|\geq |x|-|y|\geq |x|-|x|^\delta$ for $y\in D_1(x)$, $|x|>2^{1-\delta}$, one has 
\begin{align*}
	I_1(x)\leq \biggl(\frac{1+|x|^2}{1+(|x|-|x|^\delta)^2}\biggr)^{\alpha\gamma}\int_{D_1(x)} a (y)dy\to 1, \quad |x|\to\infty.
\end{align*}
Next, we evidently have, for any $|y|<\frac{|x|}{2}$, that $1+|x-y|^2\geq 1+(|x|-|y|)^2 \geq\frac{1}{4}(1+|x|^2)$; therefore,
\begin{align*}
	I_2(x)\leq 4^{\alpha \gamma } \int_{\{|y| \geq |x|^\delta\}} a (y)dy\to0, \quad |x|\to\infty.
\end{align*}
Finally, $a(y) \leq \frac{c_\alpha}{(1+ \frac{x^2}{4})^{\alpha}}$ for $y \in D_3(x)$, hence  
\begin{align*}
  I_3(x)&\leq
	c_\alpha 4^\alpha \frac{(1+|x|^2)^{\alpha\gamma}}{\Bigl(4+|x|^2\Bigr)^{\alpha}}
	\int_{D_3(x)}\frac{1}{(1+|x-y|^2)^{ \alpha\gamma}}dy\\
&\leq c_\alpha c_{\alpha\gamma} 4^\alpha \Big(\frac{(1+|x|^2)^{\gamma}}{4+|x|^2}\Big)^{\alpha}\to 0, \quad |x|\to\infty.
	\end{align*}
	where $c_\alpha$ is the normalising constant defined in \eqref{stampede}.
	As a result \eqref{eq:a_conv_a_gamma} holds. The proof is completed.
\end{proof}

\begin{lem}\label{lem:ba:suff_cond1}
	Let $a$ be defined by \eqref{stampede} with $\alpha>\frac{1}{2}$, $\gamma \in(\frac{1}{2\alpha},1)$.
Then, for any $\delta\in(0,1)$, there exists $\la=\la(\delta,\gamma)>0$, such that
\[
	(a*\omega_\la)(x) \leq (1+\delta) \omega_\la(x), \qquad x\in\R, 
\]
where
 \begin{equation}\label{eq:defofomegala}
    \omega_\la(x):=\min\bigl\{\la,a^\gamma(x)\bigr\}, \quad x\in\X.
  \end{equation}
\end{lem}

\begin{proof}
	For any $\la>0$, we define the set
  \begin{equation}\label{eq:defofsetOmegala}
    \Omega_\la:=\Omega_\la(\gamma):=\bigl\{x\in\X: a^\gamma(x)<\la\bigr\}.
  \end{equation}
 	By \eqref{eq:defofomegala}, for an arbitrary $\la>0$, we have $\omega_\la(x)\leq \la$, $x\in\X$; then $(a*\omega_\la)(x)\leq \la$, $x\in\X$, as well. In particular, cf.~\eqref{eq:defofomegala},
\begin{equation}\label{eq:ineq1213}
    (a*\omega_\la)(x)\leq
    \omega_\la(x), \quad x\in\X\setminus\Omega_\la.
  \end{equation}
	Next, by Lemma \ref{lem:a_conv_agamma_equiv_agamma}, for any $\delta>0$ there exists $\la=\la(\delta)\in(0,1)$ such that
	\[
		\sup_{x\in\Omega_\la}\dfrac{( a * a^\gamma)(x)}{a^\gamma(x)}\leq 1+\delta,
	\]
in particular,
\begin{equation*}
    (a* a^\gamma)(x)\leq (1+\delta) a^\gamma(x)=(1+\delta) \omega_\la(x), \quad x\in\Omega_\la.
\end{equation*}
  Therefore, for all $x\in\Omega_\la$,
  \begin{equation}\label{eq:ineq1214}
        (a*\omega_\la)(x)=(a*a^\gamma)(x)-\left( a*(a^\gamma-\omega_\la) \right)(x)\leq (1+\delta) \omega_\la(x),
  \end{equation}
  where we used the obvious inequality: $a^\gamma\geq \omega_\la$. By \eqref{eq:ineq1213} and \eqref{eq:ineq1214}, one gets the statement.
\end{proof}	

For a  function $\varpi:\X\to(0,+\infty)$, we define, for any $f:\X\to\R$,
\begin{equation}\label{eq:defofomeganorm}
  \lVert f\rVert_\varpi:=\sup_{x\in\X} \frac{|f(x)|}{\varpi(x)}\in[0,\infty].
\end{equation}

\begin{prop}[\protect{cf.~\cite[Propostion 3.1]{FKT2018}}] \label{prop:ba:bound_above_gen}
	Let $a$ be defined by \eqref{stampede} with $\alpha>\frac{1}{2}$, function $\varpi:\X\to(0,+\infty)$ be such that $a*\varpi$ is well-defined (for example, let $\varpi$ be bounded) and, for some $\nu\in(0,\infty)$,
\begin{equation} \label{eq:ba:Phi_est_suff_cond}
	\dfrac{(a*\varpi)(x)}{\varpi(x)}\leq\nu, \quad x\in\X.
\end{equation}
 Let $0\leq u_0\in L^\infty(\X)$ and $\normom{u_0}<\infty$; let $u=u(x,t)$ be the corresponding solution to \eqref{eq:basic_meso}. Then
\begin{equation} \label{eq:ba:u_est_above_goal}
  \normom{u(\cdot,t)}\leq \normom{u_0}e^{\nu t},\quad  t\geq0.
\end{equation}
\end{prop}
\begin{proof}
For any $f:\X\to\R_+$, with $\|f\|_\varpi<\infty$, we have
\begin{align} \label{eq:ba:conv_ineq}
	\dfrac{\min\{(a*f)(x),1\}}{\varpi(x)} \leq&\, \dfrac{(a*f)(x)}{\varpi(x)} \leq \int_{\X}\dfrac{a(y)\varpi(x-y)}{\varpi(x)}\dfrac{|f(x-y)|}{\varpi(x-y)}dy \nonumber \\ 
	\leq&\, \dfrac{a*\varpi(x)}{\varpi(x)} \normom{f}.
\end{align}
By Proposition \ref{prop:exist_uniq} and \eqref{eq:Phi_def}, for any $0\leq\tau<\Tau$, we have that 
\[
	u(x,t)=(\Phi u)(x,t), \quad t\in[\tau,\Tau].
\]
where $\Phi = \Phi_{u(\cdot,\tau)}$.
Suppose that for some $\tau=(N-1)\delta$, $\delta\in(0,1)$, $N\in\N$, we have $\normom{u_\tau}\leq \normom{u_0}e^{\nu \tau}$.
Take any $v\in C([\tau,\Tau], L^\infty(\R,\R_+))$, $t\in [\tau,\Tau]$, $\Tau:=\tau+\delta$, $0\leq u_\tau\in L^\infty(\R)$ such that
\begin{equation} \label{eq:ba:v_est}
	\normom{v(\cdot,t)}\leq \normom{u_0}e^{\nu t},\quad t\in[\tau,\Tau].
\end{equation}
We will check the following inequality
\[
  \normom{(\Phi v)(\cdot,t)}\leq \normom{u_0}e^{\nu t}, \quad t\in[\tau,\Tau].
\]
By \eqref{eq:Phi_def}, \eqref{eq:ba:conv_ineq}, \eqref{eq:ba:v_est}, one gets, for $t\in[\tau,\Tau]$,
\begin{align*}
	0&\leq \dfrac{(\Phi v)(x,t)}{\varpi(x)} \leq \dfrac{u_\tau(x)}{\varpi(x)} + \int_\tau^{t}\frac{(a*v)(x,s)}{\varpi(x)}ds\\ 
	&\leq \normom{u_0} e^{\nu\tau} + \normom{u_0} \int_{\tau}^{t} \nu e^{\nu s}ds = \normom{u_0}e^{\nu t}.\end{align*}
Since, by the proof of Proposition \ref{prop:exist_uniq}, $u$ is the limiting function for the sequence $\Phi^n v$, $n\in\N$, and $u_\tau(x) = u(x,\tau)$, one gets the statement. 
\end{proof}

\begin{prop}\label{prop:est_above}
	Let $a$ be defined by \eqref{stampede} with $\alpha>\frac{1}{2}$, $u_0\in\EEp$, and $u$ is the corresponding solution to \eqref{eq:basic_meso}. Then for any $\eps>0$ the following statements hold,
	\begin{enumerate}
  	\item If $u_0(x) \preceq a(x)$ for $x\in\RR$, then there exists $t_0$, such that for all $t\geq t_0$, 
  		\begin{equation}\label{eq:est_u_0_integrable}
				u(x,t) \preceq e^{-\frac{\eps t}{2}},\quad  x\in (-\infty, -e^{\frac{1+\eps}{2\alpha}t})\cup(e^{\frac{1+\eps}{2\alpha}t},\infty).
  		\end{equation}
		\item If $u_0(x) \preceq \inte{x}{\infty} a(y)dy$ for $x\in\RR$, then there exists $t_0$, such that for all $t\geq 0$, 
  		\begin{equation}\label{eq:est_u_0_decreasing}
				u(x,t) \preceq e^{-\frac{\eps t}{2}},\quad  x\in (e^{\frac{1+\eps}{2\alpha-1}t}, \infty).
  		\end{equation}
  \end{enumerate}
\end{prop}
\begin{proof}
	We start with proving the first statement. Recall that $\omega_\la(x) = \min\{a^\gamma(x),\la\},\ x\in\R$, for $\gamma\in (\frac{1}{2\alpha},1)$.
	By Lemma \ref{lem:ba:suff_cond1} and Proposition \ref{prop:ba:bound_above_gen}, for any $\delta\in(0,1)$ there exists $\la>0$  such that, for $\omega := \omega_\la$,
	\[
		u(x,t) \leq \normom{u_0}e^{(1+\delta)t} \min\{a^\gamma,\la\}, \qquad x\in\R,\ t\geq0.
	\]
	Then for $t_0$, such that $a^\gamma(e^{\frac{1+\eps}{2\alpha}t_0})\leq\lambda$, and for all $t\geq t_0$, $|x|\geq e^{\frac{1+\eps}{2\alpha}t}$,
	\[
		u(x,t) \leq c_\alpha \normom{u_0} \frac{e^{(1+\delta)t}} { \big(1+e^{\frac{1+\eps}{\alpha}t}\big)^{\alpha\gamma} } \leq c_\alpha \normom{u_0} e^{(1+\delta-\eps\gamma-\gamma)t}.
	\]
	where the first inequality holds by \eqref{stampede}.
	Hence it suffices to choose $\gamma \in( \frac{1}{\min\{2,2\alpha\}},1)$, $\delta\in(0, \eps(\gamma-\frac{1}{2}))$ and redefine $t_0$ such that $c_\alpha \normom{u_0} e^{(1+\delta-\eps\gamma-\gamma+\frac{\eps}{2})t_0}\leq 1$.

	To prove the second statement we note that, by Lemma \ref{lem:ba:suff_cond1}, for any $\delta\in(0,1)$, there exists $\la>0$, such that for $\omega_\la(x) = \min\{\la, a^\gamma(y)\}$, $\omega(x) = \int[x][\infty]\omega_\la(y)dy$, 
	\[
		(a*\omega)(x) = \int[x][\infty] (a*\omega_\la)(y)dy \leq (1+\delta) \int[x][\infty]\omega_\la(y)dy = (1+\delta) \omega(x), \qquad x\in\R.
	\]
	Hence Proposition \ref{prop:ba:bound_above_gen} may be applied. The rest of the proof is analogous to the first part. The proof is completed.
\end{proof}
%

Now we can prove the main result.
\begin{proof}[Proof of Theorem \ref{thm:est_front}]

	We prove the first part of the theorem.
	Let $v$ solve \eqref{eq:basic_meso} with $v(x,0) =v_0(x) = \min\{u_0,\frac{1}{2}\}$.
	By Proposition \ref{prop:compar_princ}, for fixed $t_0\in(0,T)$, $T:= \inf \{t:\|v(\cdot,t)\|_\infty\geq1\}$,
	\[
		(a*v_0)(x) \preceq \sum_{j\geq0} \frac{t_{0}^j A^j}{j!} v_0(x) = v(x,t_0) \leq u(x,t_0), \quad x\in\R,
	\]
	where $Af:=a*f$.
	Hence, by the first part of Lemma \ref{lem:conv_of_long_tail_and_o_small} applied to $v_0$, and since $u(x,t)$ is increasing in $t$, there exists $R>0$ such that	
	\begin{equation}\label{eq:some_est_on_u}
		|x|^{-2\alpha} \preceq u(x,t),\quad |x|\geq R,\ t\geq t_0.
	\end{equation}

	By \eqref{eq:some_est_on_u} and Lemma \ref{lem:hair_trigger_effect}, there exists $\tau_1\geq t_0$ such that 
	\[
		\min \bigl\{ 1, |x|^{-2\alpha}e^{(1-\frac{\eps}{2})(\tau_0+1)} \bigr\} \preceq u(x,\tau_1),\quad x\in\RR,
	\]
	where $\tau_0$ is defined in Lemma \ref{lem:subsol}.
	Hence, by Proposition \ref{prop:compar_princ} and Lemma \ref{lem:subsol}, there exits $\la\in(0,1)$, such that 
	\[
		\la g(x,t+\tau_0) = \lim_{l\to0} \frac{1}{l} \int[t][t+l] \la g(x,s+\tau_0) ds = \lim_{l\to0} \la G(x,t+\tau_0,l) \leq  u(x,t+\tau_1),\quad x\in\RR,\ t\geq0,
	\]
	where $g$ is defined by \eqref{eq:def_g_and_h} with $\frac{\eps}{2}$ instead of $\eps$ and we used, by the monotonicity of $g$ in $t$, $\la G(x,\tau_0,l)\leq \la g(x,\tau_0+1) \leq u(x,\tau_1)$, $x\in\R$, $l\in(0,1)$.

	By Lemma \ref{lem:hair_trigger_effect} and \eqref{eq:Z_Q_are_commutative}, for any $n>0$ there exists $t_n$ such that $u_0(x) \geq \la$, for $x\in B_1(x_0)$, yields $u(x,t+t_n)\geq n$, for $x\in B_1(x_0)$, $t\geq0$.
	Hence, for $t\geq \frac{2-2\eps}{\eps}(\tau_1+t_n)$,
	\[
		u(x,t+\tau_1+t_n) \geq n,\quad x\in \{x:|x|^{-2\alpha} e^{(1-\eps)(t+\tau_1+t_n)}\geq 1\}, 
	\]
	since $\{x:|x|^{-2\alpha} e^{(1-\eps)(t+\tau_1+t_n)}\geq 1\} \subset \{x:\la g(x,t+\tau_0)\geq \la\} =\{x:|x|^{-2\alpha} e^{(1-\frac{\eps}{2})t}\geq 1\}$.
	On the other hand by Proposition \ref{prop:est_above} there exits $\tau \geq t_n + \tau_1$ such that 
	\[
		u(x,t) \leq \frac{1}{n},\quad x\in \{x:|x|^{-2\alpha} e^{(1+\eps)t}\leq 1\}.
	\]
	As a result \eqref{eq:front_position_u_0_integrable} is proved. 

	Let us prove \eqref{eq:front_position_u_0_decreasing}.	
	Let $v$ solve \eqref{eq:basic_meso} with $v_0(x) := v(x,0)\not\equiv 0$ such that $v_0 \in C^\infty(\R)$ is decreasing and $v_0 \leq \min\{u_0,\frac{1}{2}\}$. As before, 
	\[
		(a*v_0)(x) \preceq v(x,t_0) \leq u(x,t_0), \quad x\in\R. 
	\]
	Similarly to \eqref{eq:some_est_on_u}, by the second part of Lemma \ref{lem:conv_of_long_tail_and_o_small},
	\begin{equation}\label{eq:one_more_est_on_u}
		x^{-2\alpha+1} \preceq u(x,t), \qquad x\geq R,\ t\geq t_0.
	\end{equation}
	By Corollary \ref{cor:Z_Q_are_commutative} and since $v_0$ is decreasing, then $v(\cdot,t)$ is decreasing in $x$, for all $t\geq0$. Therefore by Proposition \ref{prop:compar_princ} and Lemma \ref{lem:hair_trigger_effect}, for any $r\in\R$,
	\begin{equation}\label{eq:hair_trigger_monotone}
		\infty = \lim_{t\to\infty} \inf_{x\leq r} v(x,t) \leq \lim_{t\to\infty} \inf_{x\leq r} u(x,t).
	\end{equation}
	By \eqref{eq:one_more_est_on_u} and \eqref{eq:hair_trigger_monotone} there exists $\tau_1\geq t_0$, such that
	\[
		\1_{\R_-}(x) + \min\{ 1, x^{-2\alpha+1}e^{(1-\frac{\eps}{2})(\tau_0+1)} \1_{\R_+}(x) \}\preceq u(x,\tau_1), \quad x\in\R.
	\]
	Hence,
	\[
		\la h(x,t+\tau_0) \leq u(x,t+\tau_1),\quad x\in\RR,\ t\geq0,
	\]
	where $h$ is defined by \eqref{eq:def_g_and_h} with $\frac{\eps}{2}$ instead of $\eps$.
	The rest of the proof runs as before.
\end{proof}

\section*{Acknowledgements}

Pasha Tkachov (PT) wishes to express his gratitude to the 
``Bielefeld Young Researchers'' Fund for the support through the Funding Line Postdocs: ``Career Bridge Doctorate -- Postdoc''.
VB and  Tyll Krueger (TK)  are grateful for the support of the 
ZIF Cooperation group "Multiscale modelling of tumor evolution, progression and growth",
and
Wroclaw University of  Science and Technology, Faculty of Electronics.
TK is grateful for the support of the University of Verona 
during his visit.
TK is also supported by the National Science Center in Poland (NCN)
through grant 2013/11/B/HS4/01061: ``Agent based modeling of innovation diffusion''.
The authors would like to thank the anonymous referee for the valuable comments which helped to
improve the manuscript.

\bibliographystyle{alphaSinus}
\bibliography{Sinus}

\end{document}